\documentclass[11pt]{amsart}

\usepackage{amsmath, amssymb, bm}  
\usepackage{amscd, tikz-cd} 
\usepackage{enumitem} 
\usepackage{adjustbox}
\usepackage{amsthm} 
\usetikzlibrary{intersections} %

\newtheorem{Theorem}{Theorem}[section]
\newtheorem{Proposition}[Theorem]{Proposition}
\newtheorem{Lemma}[Theorem]{Lemma}

\newtheorem{Corollary}[Theorem]{Corollary}

\theoremstyle{definition}

\newtheorem{Definition}[Theorem]{Definition}

\theoremstyle{remark}

\newtheorem{Remark}[Theorem]{Remark}

\newcommand{\real}{\mathbb{R}}
\newcommand{\complex}{\mathbb{C}}

\newcommand{\PO}{\mathrm{PO}}
\newcommand{\po}{c}

\newcommand{\Tr}{\mathrm{Tr}}
\newcommand{\cH}{\mathcal{H}}
\newcommand{\loc}{\mathrm{loc}}
\newcommand{\cL}{\mathcal{L}}

\newcommand{\Det}{\mathrm{Det}}
\newcommand{\att}{\mathcal{A}}
\newcommand{\cU}{\mathcal{U}}
\newcommand{\disk}{\mathbb{D}}
\newcommand{\hf}{\hat{f}}
\newcommand{\hF}{\hat{F}}

\newcommand{\hL}{\hat{L}}

\newcommand{\hg}{\hat{g}}
\newcommand{\der}{\bm{d}}
\newcommand{\hder}{\hat{\der}}
\newcommand{\derd}{\mathcal{D}}

\newcommand{\bbJ}{\mathbb{J}}

\newcommand{\bs}{\mathbf{s}}
\newcommand{\bh}{\mathbf{h}}
\newcommand{\bsigma}{\boldsymbol{\sigma}}

\newcommand{\frg}{\mathfrak{g}}
\newcommand{\cJ}{\mathcal{J}}
\newcommand{\frL}{\mathfrak{L}}

\newcommand{\weight}{\mathcal{W}}

\begin{document}

\title[Cohomological theory]{On cohomological theory \\of dynamical zeta functions}
\author[M.~Tsujii]{Masato TSUJII}
\address{Department of Mathematics, Kyushu University, Fukuoka, 819-0395}
\email{tsujii@math.kyushu-u.ac.jp}
\keywords{dynamical zeta functions, transfer operator, leaf-wise cohomology}
\thanks{This work was supported by JSPS Grant-in-Aid for Scientific Research (B) 15H03627}
\date{\today}

\begin{abstract}
We discuss about the conjectural cohomological theory of dynamical zeta functions \cite{Juhl, Guillemin, Patterson} in the  case of general Anosov flows. Our aim is to provide a functional-analytic framework that enables us to justify  the basic part of the theory rigorously. We show that the zeros and poles of a class of dynamical zeta functions, including the semi-classical (or Gutzwiller-Voros) zeta functions, are interpreted as eigenvalues of the generators of  some transfer operators acting on the leaf-wise de Rham cohomology spaces of the unstable foliation. 
\end{abstract}

\maketitle

\section{Introduction}
For an Anosov diffeomorphism $f:M\to M$ on a closed manifold $M$, we can count its periodic points by using Lefschetz fixed point formula \cite{MR1503373}: the number of the periodic points with period $p$ is given as the alternating sum of the traces of the actions of $f^p$ on the cohomology spaces of $M$. As a consequence, the Artin-Mazur zeta function for an Anosov diffeomorphism is a rational function and  its singularities, {\it i.e.} zeros and poles, enjoy the symmetry given by Poincar\'e duality. (See  \cite{Smale}.) 
Roughly speaking, the cohomological theory of dynamical zeta functions \emph{for flows} is a quest for the counterparts of such facts in the case of Anosov flows $f^t:M\to M$. Since the actions of the flow on the cohomology spaces is trivial, this is a quite non-trivial problem from the beginning.  
One possible idea that we can find in the literature (\cite{Juhl, Guillemin, Patterson}) is to consider the actions of the flow $f^t$ on the leaf-wise de Rham cohomology spaces $\cH^m(\mathcal{F}_u)$ of its unstable foliation $\mathcal{F}_u$. Indeed, by some heuristic arguments, we can find a conjectural relation
\[
\sum_{m=0}^{\dim \mathcal{F}_u}(-1)^m \Tr ((f^{-t})^*:\cH^m(\mathcal{F}^u)\circlearrowleft))=
\sum_{\po\in \PO}\sum_{n=1}^{\infty}\frac{|c|\cdot \delta(t-n|\po|)}{|\det(\mathrm{Id}-(P^s_c)^{-n})|}
\]
where the left-hand side is an alternating sum of the traces of the actions of the flow on the leaf-wise cohomology spaces while  $\PO$  on the right-hand side denotes the set of prime periodic orbits and $|\po|$ and $P^s_{\po}$ denote respectively the period and the transversal Jacobian matrix of $\po\in \PO$ restricted to the stable subspace. (We will give precise definitions in the next section.) 
Then, in parallel to the case of Anosov diffeomorphism, this leads to another conjectural statement that the following dynamical zeta function
\begin{align*}
\zeta(s)&=\exp\left(\int \sum_{\po\in \PO}\sum_{n=1}^{\infty}
\frac{1}{t}\frac{|c|\cdot e^{-st}}{|\det(\mathrm{Id}-(P^s_c)^{-n})|}\cdot \delta(t-n|\po|) dt \right)\\
&=\exp\left( \sum_{n=1}^{\infty}\sum_{\po\in \PO}
\frac{e^{-sn|\po|}}{n\cdot |\det(\mathrm{Id}-(P^s_c)^{-n})|} \right),
\end{align*}
extends  to a meromorphic function  on  the complex plane and its singularities correspond to the eigenvalues of the generators of the push-forward action $(f^{-t})^*:\cH^m(\mathcal{F}_u)\circlearrowleft$. Further it is expected under some additional assumptions that the major part of the zeros of $\zeta(s)$ comes from the action of the flow on the bottom cohomology space $\cH^0(\mathcal{F}_u)$ and other singularities are only minor.  We refer the introduction part of the book \cite{Juhl} for the circle of ideas related to the cohomological theory of dynamical zeta functions. 

The geodesic flows on hyperbolic closed surfaces are motivating examples for the cohomology theory mentioned above. The dynamical zeta function $\zeta(s)$ above for such flows coincide with the Selberg zeta functions \cite{McKean72} up to shift of the variable $s$ by $1$ and, for the latter, we have precise information on analytic properties and the structure behind them, by Selberg's trace formula and the method using representation theory of $SL(2,\real)$. Hence we are able to check validity of the cohomological theory mentioned above. (See \cite{Guillemin, Patterson}.) Further we can extend the results to the case of geodesic flows on more general classes of homogeneous spaces. (See the book \cite{Juhl} and the references there-in.) However the methods used in such arguments depend crucially on the homogeneous property of the spaces and will not be valid in more general cases, such as those of the geodesic flows on closed manifolds with negative variable curvature. 
 
The purpose of the present paper is to propose a more direct approach to the cohomological theory of dynamical zeta function for flow so that it is applicable to the general setting of Anosov flows. Recently there are a few developments in functional-analytic methods on smooth hyperbolic dynamical systems (\cite{Blank02, BaladiTsujii07, GL08, ButLiv2007,  FaureSjostrand11, GLP12, DZ16}). In particular, we now understand that the generators of the transfer operators associated to Anosov flows exhibit discrete spectra (called Ruelle-Pollicott resonances) in general \cite{ButLiv2007,FaureSjostrand11} and that the asymptotic distributions of the periods of periodic orbits are controlled by such discrete eigenvalues of related transfer operators \cite{DZ16}. It should be a natural idea to put the problems around the cohomological theory of dynamical zeta function in the light of such developments. 

The main results of this paper, Theorem \ref{th:main}, is presented in the next section and provides a functional-analytic framework that enables us to justify the basic part of the cohomological theory mentioned above in the general case of $C^\infty$ Anosov flows. 
The main technical problem behind (the proof of) the theorem is that the unstable foliation is not necessarily smooth and, consequently, nor are the objects related to the cohomological theory, such as transfer operators and exterior derivative operators along unstable leafs. Indeed this is the problem that appears first when we try to justify the cohomological theory in more general cases of (non-homogeneous) Anosov flows.  
We will resolve (or avoid) the difficulties by constructing non-smooth embedding of the flow $f^t:M\to M$ into another smooth flow $\tilde{f}^t:\tilde{M}\to \tilde{M}$ on an extended space $\tilde{M}$ and showing that the problematic non-smooth transfer operators on $M$ are realized as a restriction of smooth transfer operators associated to $\tilde{f}^t$. This kind of method is originated in the paper \cite{CviVat93} of Cvitanovi\'c and Vattay and used more recently in \cite{GL08,MR3648976} to resolve the problems caused by non-smoothness of the unstable foliation. In this paper, we  enhance  the method in order to deal with the exterior derivative operators along unstable leafs as well as transfer operators. The detailed explanation will be given in Section \ref{sec:ext}. 

The cohomological theory of dynamical zeta function has been restricted to the case of homogeneous flows in the previous papers \cite{Juhl, Guillemin, Patterson} and remained in a conjectural level in more general cases.
From this viewpoint, the significance of the result in this paper should be clear. 
But, admittedly, in order to be convinced of its usefulness, we have to study the action of the transfer operators on the leaf-wise cohomology spaces by using the proposed framework. We will give a brief discussion about this at the end of Section~\ref{sec:def} and present an interesting observation by comparing our main result with that of the previous paper \cite{MR3648976}. Also we will present a conjectural picture that we expect in Remark \ref{rem:conj}. But a full-scale study is deferred to future works.


\section{Definitions and The main result}\label{sec:def}
In this section, we present the main results of this paper after giving basic definitions and related preliminary arguments. 
\subsection{Anosov flow}
Let $(M,g)$ be a $C^\infty$ closed Riemann manifold. We consider a $C^\infty$ flow $f^t:M\to M$ generated by a vector field $v_f:M\to TM$. 
\begin{Definition}[Anosov flow]\label{def:Af}
A $C^\infty$ flow $f^t:M\to M$ is said to be  an \emph{Anosov flow} if there exists a $Df^t$-invariant continuous splitting 
\begin{equation}\label{eq:hyp_decomp}
TM = E_0 \oplus E_s \oplus E_u
\end{equation}
of the tangent bundle $TM$ such that $E_0$ is the one-dimensional subbundle spanned by the (non-vanishing) generating vector field $v_f$ and the other two subbundles $E_s$ and $E_u$ are uniformly expanding and contracting respectively in the sense that, for some constant $C>0$ and $\chi>0$, we have 
\begin{align*}
&\|Df^t|_{E_s}\|_g\le \exp(-\chi t +C) \quad \text{for $t\ge 0$}
\intertext{and}
&\|Df^{-t}|_{E_u}\|_g\le \exp(-\chi t +C) \quad \text{for $t\ge 0$.}
\end{align*}
\end{Definition}
The constant $\chi>0$ will be called the hyperbolicity exponent of the Anosov flow $f^t$ (though there is some freedom in its choice). The subbundles $E_s$ and $E_u$ are called the stable and unstable subbundle respectively. We set $d_s=\dim E_s$ and $d_u=\dim E_u$ so that
\[
d:=\dim M=d_s+d_u+1.
\] 
Recall the stable manifold theorem \cite{PS1, PS2} and the following  definition. 
\begin{Definition}
The integral manifold of the subbundle $E_u$ (resp. $E_0\oplus E_u$) passing through a point $p\in M$ is  called the unstable manifold (resp. center-unstable manifold of $p$ and denoted by $W^u(p)$ (resp.  $W^{cu}(p)$). These are $C^\infty$ immersed submanifolds in~$M$. A small open topological disks around $p$ in $W^u(p)$ (resp. $W^{cu}(p)$) is called the local unstable  manifold (resp. local center-unstable manifold). The foliation that the  unstable manifolds form is  called \emph{the  unstable  foliation} and denoted by  $\mathcal{F}_u$. This foliation is not necessarily smooth, though each of the leafs is smooth.  
\end{Definition}

We are most interested in the following subclass of Anosov flows.  
\begin{Definition}[Contact Anosov flow] An Anosov flow $f^t:M\to M$ is said to be a contact Anosov flow if the manifold $M$ is odd dimensional, say $\dim M=2\ell+1$ for some integer $\ell\ge 1$, and if it preserves a contact form $\alpha$, which is by definition a differential $1$-form on $M$ satisfying the complete non-integrablity condition, {\it i.e.} $\alpha\wedge (d\alpha)^{\ell}$ vanishes nowhere.  The geodesic flow on a closed Riemann manifold with negative sectional curvature is a prominent example of contact Anosov flow.
\end{Definition}

 If $f^t:M\to M$ is a contact Anosov flow, 
the contact form $\alpha$ induces a symplectic form $\omega=d\alpha$ on the null space $\ker \alpha$ of $\alpha$, which is also preserved by the flow. Since the flow $f^t$ is contracting and expanding on $E_s$ and $E_u$ respectively,  $E_s$ and $E_u$ are Lagrangian subspaces of $\ker \alpha$ with respect to the symplectic form $\omega$. In particular, we have $
\ker \alpha= E_s\oplus E_u$ and $\dim E_s=\dim E_u=\ell$ in the case of contact Anosov flow.

\subsection{Periodic orbits}
We henceforth assume that $f^t:M\to M$ is a $C^\infty$ Anosov flow. 
We write $\PO$ for the set of prime periodic orbits of the flow $f^t$. 
For $c\in \PO$, we write $|c|$ for its prime period. Since we have $f^{|\po|}(p)=p$ for each point $p$ on $\po$, 
the action of $Df^{|\po|}$ induces a linear transform on the tangent space $T_pM$ at $p$. This linear transform preserves the splitting 
\[
T_pM =  E_0(p) \oplus E_u(p) \oplus E_s(p).
\]
By definition, the linear transform $Df_p^{|\po|}$ is  identical, contracting and expanding on $E_0(p)$, $E_s(p)$ and $E_u(p)$ respectively. The restriction of   $Df_p^{|\po|}$ to $E_u(p)\oplus E_s(p)$ is called the transversal Jacobian matrix of $\po$ and denoted by
\[
P_\po:E_u(p)\oplus E_s(p) \to E_u(p)\oplus E_s(p).
\]
Its restriction to the subspaces  $E_s(p)$ and $E_u(p)$ are denoted respectively by 
\[
P^s_{\po}:=P_\po|_{E_s(p)}:E_s(p)\to E_s(p)\quad\text{and}\quad
P^u_{\po}:=P_\po|_{E_u(p)}:E_u(p)\to E_u(p).
\]
\begin{Remark} \label{rem:sim}
The equivalence class with respect to similarity (or conjugacy) of the transversal Jacobian matrix $P_\po$  does not depend on the choice of the point $p$ on $c$. The same is true for the linear transforms $P^s_\po$ and $P^u_\po$. Thereby we drop dependence on $p$ from the notation.
\end{Remark} 

\subsection{The semi-classical zeta functions}\label{ss:scz}
In the case of  a contact Anosov flow $f^t:M\to M$, the semi-classical (or Gutzwiller-Voros) zeta function is a function of one complex variable $s$ defined formally by 
\begin{equation}\label{eq:def_sz}
\zeta_{sc}(s)=\exp\left(
 -\sum_{\po\in \PO} \sum_{n=1}^\infty
 \frac{e^{-sn|\po|}}
 {n\cdot \sqrt{|\det(\mathrm{Id}-(P_\po)^{-n})|}}
 \right).
\end{equation}
The infinite sum on the right-hand side converges absolutely if the real part of $s$ is sufficiently large. Hence this definition makes sense and  gives a non-vanishing holomorphic function on the region $\Re(s)>C$ for some $C>0$. As we will explain later, this holomorphic function extends to a meromorphic function on the complex plane $\complex$.  

We extend the definition \eqref{eq:def_sz} to the general setting of  Anosov flows. Let $f^t:M\to M$ be a $C^\infty$ Anosov flow and let $\pi_G:G\to M$ be the Grassmann bundle that consists of $d_u$-dimensional subspace of $TM$. Then the flow $f^t$ induces a $C^\infty$ flow on $G$,
\begin{equation}\label{eq:fG}
\hf^t_G:G\to G, \quad \hf^t_G(x,\sigma)=(f^t(x), Df^t(\sigma))
\end{equation}
where $\sigma$ denotes a $d_u$-dimensional subspace in $T_xM$. Note that the  unstable subbundle $E_u$  gives a section  
\begin{equation}\label{eq:iotau}
\iota_u:M\to G, \quad \iota_u(x)=E_u(x)\in G,
\end{equation}
which is equivariant for the actions of the flow in the sense that the diagram 
\begin{equation}\label{cd:ftg2}
\begin{CD}
G @>{\hf^t_G}>> G\\
@A{\iota_u}AA @A{\iota_u}AA\\
M@>{f^t}>> M
\end{CD}
\end{equation} 
commutes. We consider a $C^\infty$ line bundle $\pi_{\hL}:\hL\to G$ over $G$  and   a $C^\infty$ one-parameter group of  vector bundle maps of $\hL$ over the flow $\hf_G^t$,
\begin{equation}\label{eq:hg}
\hg^t:\hL\to \hL\quad\text{such that } \pi_{\hL}\circ \hg^t=\hf_G^t\circ  \pi_{\hL}.
\end{equation}
We write $\pi_L:L=\iota_u^* \hL\to M$ for the continuous line bundle obtained as the pull-back of $\hL$ by the section $\iota_u$. Then $\hg^t$ induces the continuous one-parameter  group of vector bundle maps $
g^t:L\to L$ which makes the following diagram commutes:
\[
\begin{CD}
\hL @>{\hg^t}>>\hL\\
@VV{(\iota_u)^*}V @VV{(\iota_u)^*}V\\
L @>{g^t}>> L
\end{CD}
\]
where $(\iota_u)^*$ denotes the pull-back operation by the section $\iota_u$. 

\begin{Definition}
In the setting as above, the generalized semi-classical zeta function $\zeta(s)$ is the function of one complex variable $s$ defined formally by 
\begin{equation}\label{eq:def_gsz}
\zeta(s)=\exp\left(
 -\sum_{\po\in \PO} \sum_{n=1}^\infty\frac{e^{-sn|\po|}\cdot g_\po^n}{n\cdot |\det(\mathrm{Id}-(P^s_\po)^{-n})|}\right)
\end{equation}
where $g_{\po}$ is the real number which represents the linear transform that $g^{|\po|}$ induces on the fiber of the line bundle $L$ at each point on  $\po$.
\end{Definition}
Let us see that the semi-classical zeta function \eqref{eq:def_sz} is a special case of the definition above.  
Suppose that $\hg^t:\hL\to \hL$ is the natural action of $f^t$ on the $1/2$-density line bundle of the tautological vector bundle $V_G$ on $G$. Then the line bundle $L:=\iota_u^* \hL$ is the $1/2$-density line bundle of $E_u$ and $g^t:L\to L$ is the natural push-forward action of $f^t$ on $L$. If we suppose further that $f^t:M\to M$ is a contact Anosov flow, we see that  $P^u_\po$ is similar to ${}^t(P^s_{\po})^{-1}$ and, by algebraic computation, that the function $\zeta(s)$ in \eqref{eq:def_gsz} coincides with the semi-classical (or Gutzwiller-Voros) zeta function $\zeta_{sc}(s)$ in \eqref{eq:def_sz}.

\subsection{Transfer operator and dynamical Fredholm determinant}
Let $\pi_V:V\to M$ be a continuous  finite dimensional complex vector bundle over $M$. (Note that we do not assume smoothness of $V$.) We consider a continuous one-parameter semi-group of vector bundle isomorphisms
\[
F^{t}_V: V\to V\quad\text{for $t\ge 0$}
\]  
over the Anosov flow $f^{t}$. A simple example of such semi-group is the restriction of $Df^t$ to $V=E_u$.  But we consider a little more involved cases. 

By definition, the vector bundle map $F^{t}_V$ induces a linear isomorphism from the fiber $\pi_{V}^{-1}(x)$ at a point $x\in M$ to the fiber $\pi_{V}^{-1}(f(x))$ at  $f(x)\in M$. In particular, if $p$ is a point on a periodic orbit $\po\in \PO$ and if $t=|\po|$, it induces a linear isomorphism 
\[
V_{\po}:=F^{|\po|}_V:\pi_{V}^{-1}(p)\to \pi_{V}^{-1}(p). 
\]
Note that the conjugacy class of $V_{\po}$ does not depend on the choice of the point $p$ on $\po$, so that we drop the point $p$ from the notation.

The continuous one-parameter group $F^{t}_V$ naturally induces the push-forward action on the space $\Gamma^{0}(V)$ of continuous sections of $V$: 
\begin{equation}\label{eq:cLtV}
\cL^t_V: \Gamma^{0}(V)\to \Gamma^{0}(V),\quad \cL^t_Vu(x)= F^{t}_{V}(u(f^{-t}(x))).
\end{equation}
This kind of push-forward operators acting on the sections of vector bundles will be called  (vector-valued) transfer operators. 

Below we would like to consider the traces of the transfer operators $\cL^t_V$. But it is usually not possible to apply the general argument on the trace of linear operators to  $\cL^t_V$. We thus take a roundabout way instead. 
The flat trace (or Atiyah-Bott-Guillmann trace) $\mathrm{Tr}^\flat \cL^t_V$ of 
the transfer operator $\cL^t_V$ is defined as the integration of its Schwartz kernel on the diagonal set. It is well-defined as a distribution with respect to the variable $t$ and, by computation, we find 
\[
\mathrm{Tr}^\flat \cL^t_V= \sum_{\po\in \PO} \sum_{n=1}^\infty 
\frac{|\po|\cdot \mathrm{Tr}\, V_{\po}^n}{|\det(\mathrm{Id}-P_\po^{-n})|}\cdot \delta(t-n|\po|).
\]
\begin{Remark} For the definition and computation of the flat trace $\mathrm{Tr}^\flat \cL^t_V$, we refer \cite[Lemma B1]{Dyatlov-Zworski16}. 
As far as the argument in this paper concerns, 
one can take the expression above as the definition of the flat trace $\mathrm{Tr}^\flat \cL^t_V$.
\end{Remark}

We define the dynamical Fredholm determinant of $\cL^t_V$ by
\begin{align}\label{eq:dynFD}
d(s)=d(s; \cL^t_V)&=\exp\left(-\int_{+0}^\infty \frac{e^{-st}}{t}\cdot \mathrm{Tr}^\flat \cL^t_V dt \right) \\
&= \exp\left(-\sum_{\po\in \PO} \sum_{n=1}^\infty 
\frac{e^{-sn|\po|}\cdot \mathrm{Tr}\, V_{\po}^n}{n\cdot |\det(\mathrm{Id}-P_\po^{-n})|} \right)\notag
\end{align}
where the lower bound $+0$ of  the integration denotes some positive real number smaller than the minimum of the periods of periodic orbits. 
\begin{Remark}\label{rem:heu}
By heuristic argument confusing the flat trace with the usual trace and supposing that  the generator of $ \cL^t_V$ has discrete eigenvalues $\{\rho_i\}_{i=1}^\infty$ and also that $\Re(s)\gg 0$, we surmise  that 
\begin{align*}
(\log d(s))'&=\int^\infty_0 e^{-st}\cdot\mathrm{Tr}^\flat \cL^t_V dt= \int^\infty_0 \sum_{i=1}^\infty e^{-(s-\rho_i)t}dt
\\
&=\sum_{i=1}^\infty(s-\rho_i)^{-1}=\left(\log \prod_{i=1}^\infty (s-\rho_i)\right)'.
\end{align*}
Thus the zeros of $d(s)$ is expected to coincide with the discrete eigenvalues of the generator of $ \cL^t_V$. 
\end{Remark}

\subsection{The transfer operators $\cL^t_k$}
We next introduce  a few concrete settings for the vector bundles and semi-groups of vector bundle maps discussed in the last subsection. 
Recall the line bundle $L$ and the semi-group of vector bundle maps $g^t:L\to L$ from Subsection \ref{ss:scz}. For $0\le k\le d_u$, we consider the continuous vector bundle
\begin{equation}\label{eq:Vk}
V_k= L\otimes (E_u^*)^{\wedge k}
\end{equation}
and the semi-group $F_k^{t}:V_k\to V_k$ of vector bundle maps on it defined by  
\[
F^{t}_{k}(l \otimes \omega)=  g^t(l)\otimes (Df^{-t})^*(\omega) \quad \text{for $l \otimes \omega\in V_k=L\otimes (E_u^*)^{\wedge k}$}.
\]
It induces the continuous semi-group of transfer operators
\begin{equation}\label{eq:cLtk}
\cL^t_{k}:\Gamma^{0}(V_k)\to \Gamma^{0}(V_k), \quad \cL^t_k\varphi(x)= F^{t}_{k}(\varphi(f^{-t}(x))).
\end{equation}
The dynamical Fredholm determinant $d_k(s)$ of $\cL^t_k$ is then  defined by 
\begin{equation}\label{eq:dynfd}
d_k(s)=\exp\left(
 -\sum_{\po\in \PO} \sum_{n=1}^\infty\frac{e^{-sn|\po|}\cdot g_\po^n \cdot  \mathrm{Tr}( (P_\po^u)^{-n})^{\wedge k}}{n\cdot |\det(\mathrm{Id}-P_\po^{-n})|}\right).
\end{equation}
Hence, by the algebraic relation\footnote{Notice that the absolute value of the eigenvalues of $P^u_c$ are smaller than $1$.}
\[
|\det(\mathrm{Id}-(P_\po^u)^{-n})|=\det(\mathrm{Id}-(P_\po^u)^{-n})=\sum_{k=0}^{d_u} (-1)^k\cdot  \mathrm{Tr}( (P_\po^u)^{-n})^{\wedge k},
\] 
we find that the generalized semi-classical zeta function $\zeta(s)$  in \eqref{eq:def_gsz} is expressed as the alternating product of $d_k(s)$: 
\begin{equation}\label{eq:alt_prod}
\zeta(s)=\prod_{k=0}^{\infty}d_k(s)^{(-1)^k}.
\end{equation}

\subsection{A remark on smoothness of the unstable foliation $\mathcal{F}_u$}\label{ss:remwu}
Before proceeding further, we put a remark about the smoothness of the unstable foliation  $\mathcal{F}_u$. 
As we have noted already, each leaf of the unstable foliation $\mathcal{F}_u$ is $C^\infty$,  though the foliation $\mathcal{F}_u$ itself if not necessarily smooth. Actually a little more is true. The following observations are not completely obvious but clear from the construction of the local unstable manifolds\cite{PS2}.

First observe that we can take a continuous map
\[
\mathbf{W}^{cu}:M\to C^\infty(\disk^{d_u+1}, M)
\] 
so that the image of $\mathbf{W}^{cu}(p)\in C^\infty(\disk^{d_u+1}, M)$ is a local center-unstable manifold of $p\in M$. 
This implies in particular that  $W^{cu}_{\loc}(p)$ for $p\in M$ is uniformly bounded in $C^\infty$ sense. Further we can take a continuous map
\[
\mathbf{T}_k:M\to C^{\infty}(\disk^{d_u+1}\times \real^{d_k}, V_k)\quad\text{for $0\le k\le d_u$}
\]
so that  $\mathbf{T}_k(p)\in C^{\infty}(\disk^{d_u}\times \real^{d_k}, V_k)$ is a $C^\infty$ vector bundle map from the trivial bundle $\disk^{d_u+1}\times \real^{d_k}$ to $V_k$ over the map $\mathbf{W}^{cu}(p)$. 
 
For $0\le k\le d_u$, we define
 $\Gamma^\infty_u(V_k)\subset\Gamma^0(V_k)$ as the set of continuous sections of $V_k$ whose pull-back by $\mathbf{T}_k(p)$ is $C^\infty$ for each $p\in M$ and depends on $p$ continuously in the $C^\infty$ topology. We suppose that it is equipped with the weakest topology so that, for each $p\in M$,  the pull-back of a section $u\in \Gamma^\infty_{cu}(V_k)$ by $\mathbf{T}_k(p)$ depends on $u$ continuously in $C^\infty$ sense. 
Then the operator $\cL^t_k$ in \eqref{eq:cLtk} restricts to a continuous operator
\begin{equation}\label{eq:cLtk2}
\cL^t_k:\Gamma^\infty_{cu}(V_k)\to \Gamma^\infty_{cu}(V_k).
\end{equation}

\begin{Remark}\label{rm:gammauM}
As a variation of the definition above, we define $\Gamma^\infty_{cu}(M)\subset \Gamma^0(M)$ as $\Gamma^\infty_{cu}(V_k)$ in the case where $V_k$ is the trivial line bundle.  
\end{Remark}

\subsection{Exterior derivative along the unstable leafs}\label{ss:extder}
Next we observe that the line bundle $\pi_L:L\to M$ restricted to each unstable manifold $W^u(q)$ is equipped with a unique $C^\infty$ flat connection that is invariant with respect to the action of $g^t$. Though this fact is rather well known, we reproduce the related argument here because the construction of that flat connection is important in our argument. 
First of all, note that  the section $\iota_u$ restricted to an unstable manifold $W^u(q)$ is a $C^\infty$ map into the Grassmann bundle $G$. Therefore the vector bundle $L$ restricted to each unstable manifold $W^u(q)$ is  $C^\infty$ and so is  the action of the vector bundle maps $g^t$ on it. 

Suppose that $p$ and $p'$ are two points on an unstable manifold $W^u(q)$ and that  $l$ and $l'$ are two non-zero elements in $L$ over $p$ and $p'$ respectively. Since $p$ and $p'$ belong to $W^u(q)$, the distance between $f^{-t}(p)$ and $f^{-t}(p')$ converges to $0$ exponentially fast as $t\to +\infty$ and  the ratio
$\|g^{-t}(l')\|/\|g^{-t}(l)\|$ 
converges to a non-zero value as $t\to +\infty$. 
We define the flat connection $H$ on $L$ restricted to $W^u(q)$
so that $l$ and $l'$ are a parallel transport of each other if and only if
\[
\lim_{t\to \infty} \mathrm{dist}\left(\frac{g^{-t}(l')}{\|g^{-t}(l)\|}, \frac{g^{-t}(l)}{\|g^{-t}(l)\|}\right) =0
\]
where $\mathrm{dist}(\cdot,\cdot)$ denotes a distance on $L$ compatible with its topology. This definition does not depend on the choice of the distance $\mathrm{dist}(\cdot,\cdot)$ and gives a $C^\infty$ flat connection on $L$ along each of the unstable leafs. We call it \emph{the dynamical flat connection along unstable leafs}.  We omit  the proof of existence and uniqueness of the flat connection $H$ above since it is rather standard. (See the remark below.) 
By definition, the connection $H$ is invariant with respect to the action of $g^t$ in the sense that $g^t$ carries the flat connection along $W^u(p)$ to that along $W^u(f^t(p))$. 
Note that the connection $H$ is not necessarily smooth  in the directions transversal to the unstable leafs.

\begin{Remark} One can construct the dynamical connection $H$ by a variant of the graph transform method. For the construction,  we need to assume that $L$ is one dimensional. The parallel argument will not be true in general for the case where $L$ is a higher dimensional vector bundle. Still we can generalize the argument in this paper to the case where $L$ is a higher dimensional vector bundle if we assume existence of a flat connection on it corresponding to the dynamical flat connection. For instance, one can consider the case where $L$ is a $C^\infty$ vector bundle over $M$ equipped with a $C^\infty$ flat connection and $g^t:L\to L$ is the (unique) one-parameter group of $C^\infty$ vector bundle maps over $f^t$ that preserves the flat connection. 
The results obtained will be   ``twisted" versions of those in this paper.  
\end{Remark}

By virtue of the dynamically defined connection $H$, we may trivialize the line bundle $L$ along the unstable leafs and  define the exterior derivative operator along the unstable leafs
\begin{equation}\label{def:deltau}
\der^u=\der^u_k:\Gamma^\infty_{cu}(V_k)\to \Gamma^\infty_{cu}(V_{k+1}),\quad k=0,1,2,\dots, d_u-1.
\end{equation}
Since $
\der^u_{k+1}\circ \der^u_k =0$, we have the cochain complex 
\[
\begin{CD}
\Gamma_{cu}^\infty(V_0) @>{ \der^u_0}>> \Gamma_{cu}^\infty(V_1) @>{\der^u_1}>>\cdots @>{\der^u_{d_u-1}}>> \Gamma_{cu}^\infty(V_{d_u}).
\end{CD}
\]
From the  invariance of the dynamical flat connection, the transfer operators $\cL^t_{k}$ in \eqref{eq:cLtk2} induce the cochain maps
\begin{equation}\label{eq:cochain}
\begin{CD}
\Gamma_{cu}^\infty(V_0) @>{ \der^u_0}>> \Gamma_{cu}^\infty(V_1) @>{\der^u_1}>>\cdots @>{\der^u_{d_u-1}}>> \Gamma_{cu}^\infty(V_{d_u})  \\
 @V{\cL^t_0}VV @V{\cL^t_1}VV  @. @V{\cL^t_{d_u}}VV
\\
\Gamma_{cu}^\infty(V_0) @>{\der^u_0}>> \Gamma_{cu}^\infty(V_1) @>{\der^u_1}>>\cdots @>{\der^u_{d_u-1}}>> \Gamma_{cu}^\infty(V_{d_u})  
\end{CD}
\end{equation}
This diagram of operators is at the core of the cohomological theory that we are interested in. It implies that, if some section $u$ in $\Gamma_{cu}^\infty(V_k)$ is an eigenfunction of $\cL^t_k$ (or its generator) for a certain eigenvalue, its image $\der^u_k(u)$ by the exterior derivative operator $\der^u_k$ will be the eigenfunction of $\cL^t_{k+1}$ (or its generator) for the same eigenvalue provided that if it does not vanish. This explains cancellations between the zeros of $d_k(s)$ in the alternative product \eqref{eq:alt_prod}. Further we see that the singularities of $\zeta(s)$ that survive such cancellations are related to the spectrum of the generators of the transfer operators acting on the cohomology spaces.  However a technical problem here is how we can set up appropriate Hilbert (or Banach) spaces as the completions of the spaces $\Gamma_{cu}^\infty(V_k)$ with respect to some norms so that the operators $\cL^t_k$ and $\der^u_k$ extend to bounded operators between them and, in addition, that the generators of $\cL^t_k$ on them exhibit discrete spectra. This problem is far from trivial because the operators $\cL^t_k$ and $\der^u_k$ are not smooth. Actually we already have a solution \cite{MR3648976} when we consider solely the transfer operators $\cL^t_k$, which uses the Grassmann extension of the flow. (See also \cite{GL08}.) However it is not clear whether the exterior derivative operators $\der^u_k$ extend to bounded operators between the spaces that appeared in such a solution. This is the point that we address in this paper. 

\subsection{Main result}
Now we present the main result. 
\begin{Theorem}\label{th:main}
For arbitrarily large $C>0$, there exist Hilbert spaces 
\[
\Lambda_k=\Lambda_k(C)\quad \text{for $0\le k\le d_u$}
\]
that are obtained as completions of the spaces $\Gamma^\infty_{cu}(V_k)$ in the diagram \eqref{eq:cochain} with respect to some norms, such that the following hold true:
\begin{enumerate}
\item The semi-group of transfer operators $\cL^t_k$ extends to a strongly continuous semi-group of bounded operators $
\cL^t_k:\Lambda_k\to \Lambda_{k}$
and the spectral set of its generator $A_k$  in the region $\Re(s)>-C$ consists of  discrete eigenvalues with finite multiplicity.  
\item The operator $\der^u_k$ extends boundedly to $
\der^u_k:\Lambda_k\to \Lambda_{k+1}$.
\end{enumerate}
In particular the diagram \eqref{eq:cochain} extends to that of bounded operators
\[
\begin{CD}
\Lambda_0 @>{ \der^u_0}>> \Lambda_1 @>{\der^u_1}>>\cdots @>{\der^u_{d_u-2}}>> \Lambda_{d_u-1} @>{\der^u_{d_u-1}}>> \Lambda_{d_u} \\
 @V{\cL^t_0}VV @V{\cL^t_1}VV  @. @V{\cL^t_{d_u-1}}VV@V{\cL^t_{d_u}}VV
\\
\Lambda_0 @>{ \der^u_0}>> \Lambda_1 @>{\der^u_1}>>\cdots @>{\der^u_{d_u-2}}>> \Lambda_{d_u-1} @>{\der^u_{d_u-1}}>> \Lambda_{d_u}
\end{CD}
\]
Further the dynamical Fredholm determinant ${d}_k(s)$ extends to a holomorphic function on the complex plane $\complex$ and its zeros coincide with the discrete eigenvalues of the generator $A_k$ on the region $\Re(s)>-C$ up to multiplicity. 
\end{Theorem}
From the theorem above, the operator $\cL^t_k$ induces a strongly continuous semi-group of operators 
\begin{equation}\label{eq:bbLt}
\frL^t_k : \mathbb{H}^k\to \mathbb{H}^k
\end{equation}
on the (reduced) quotient space 
\[
\mathbb{H}^k=\ker(\der^u_k:\Lambda_k\to \Lambda_{k+1}) 
/ \mathrm{closure}(\mathrm{Image}\, \der^u_{k-1}:\Lambda_{k-1}\to \Lambda_{k})
\]
where the closure is taken in the Hilbert space $\Lambda_k$. The generator $\mathbb{A}_k$ of this semi-group has discrete spectrum in the region $\Re(s)>-C$. Hence we get the following conclusion which realizes (a basic part of) the  cohomology theory mentioned  in the introduction if we take the Hilbert spaces $\mathbb{H}^k$ as the leaf-wise cohomology spaces\footnote{These are different from the leaf-wise cohomology spaces considered in topology. Though each unstable leaf is usually dense in $M$, the space $\mathbb{H}_0$ will be infinite dimensional.} along the unstable foliation $\mathcal{F}_u$. 
\begin{Corollary}
The devisor of the generalized semi-classical zeta function \eqref{eq:def_gsz}  coincides with the alternating sum
\[
\sum_{k=0}^d (-1)^k \mathrm{Div}(\mathbb{A}_k)
\] 
of the divisor $\mathrm{Div}(\mathbb{A}_k)$ of the resolvent of the generator $\mathbb{A}_k$ in  the region $\Re(s)>-C$. (The constant $C$ is that in Theorem \ref{th:main}.)    
\end{Corollary}
\begin{proof}
We fix $y_0\in \real$ arbitrarily and set $\rho_0=C+iy_0$ where $C>0$ is a arbitrary large constant in Theorem \ref{th:main}.  We consider the resolvent 
\[
\mathcal{R}_k(s)=(s-A_k)^{-1} =\int_0^\infty e^{-st}\cL^t_k dt\;: \;\Lambda_k\to \Lambda_k
\]
of the generator $A_k$ of the semi-group $\cL^t_k:\Lambda_k\to \Lambda_k$. 
The spectral set $\Sigma_k(\rho_0)$ of the resolvent $\mathcal{R}_k(\rho_0)$ in the region $|z|>(2C)^{-1}$ consists of finitely many discrete eigenvalues with finite multiplicity and are in one-to-one correspondence to the  discrete eigenvalues of the generator $A_k$ in the region $|s-\rho_0|<2C$. Let $\Pi_k:\Lambda_k\to \Lambda_k$ be the (finite rank) spectral projector of $\mathcal{R}_k(\rho_0)$ for the spectral set $\Sigma_k(\rho_0)$. Since the kernel of this projection is characterized as 
\[
\ker \Pi_k=\{ u\in \Lambda_k\mid \lim_{m\to \infty}(2C-\varepsilon)^{m} \cdot \|\mathcal{R}_k(\rho_0)^m u\|_{\Lambda_k}= 0 \text{ for any $\varepsilon>0$}\},
\]
the operator $\der^u_k$ maps the subspace $\ker \Pi_k$ into $\ker \Pi_{k+1}$. 
Consequently we obtain the following commutative diagram
\[
\begin{CD}
\Lambda_{k-1}/\ker \Pi_{k-1}@>{\der^u_{k-1}}>> \Lambda_{k}/\ker \Pi_k@>{\der^u_{k}}>> \Lambda_{k+1}/\ker \Pi_{k+1}
\\@V{\mathcal{R}_{k-1}(\rho_0)}VV@V{\mathcal{R}_k(\rho_0)}VV @V{\mathcal{R}_{k+1}(\rho_0)}VV\\
\Lambda_{k-1}/\ker \Pi_{k-1}@>{\der^u_{k-1}}>> \Lambda_{k}/\ker \Pi_k@>{\der^u_{k}}>> \Lambda_{k+1}/\ker \Pi_{k+1}
\end{CD}
\]
where, by slight abuse of notation,  we write $\der^u_{k}$ and $\mathcal{R}_{k}(\rho_0)$ for the respective induced operators on the quotient spaces.  In this diagram, the spaces are all finite dimensional and we may replace $\mathcal{R}_{k}(\rho_0)$ by $\cL^t_k$ and also by $A_k$. Hence it is now easy to conclude the claim of the corollary because any eigenvalue $s$ of the generator $A_k$ with $\Re(s)>-C$ appears as an eigenvalue of $\mathcal{R}_k(\rho_0)$ contained in $\Sigma_k(\rho_0)$ if we set $y_0=\Im(s)$ in the argument above.
\end{proof}

\subsection{The case of contact Anosov flows}
We finish this section by discussing about the case of contact Anosov flows where we are most interested in. 
Let us suppose that $f^t:M\to M$ is a contact Anosov flow and the one parameter group of vector bundle maps $\hg^t:\hL\to \hL$ is that considered at the end of Subsection \ref{ss:scz}. Then, as we have noted,   the semi-classical zeta function $\zeta_{sc}(s)$ in \eqref{eq:def_sz} coincides with  the generalized semi-classical zeta function \eqref{eq:def_gsz} and therefore we can apply the results in the last subsection to $\zeta_{sc}(s)$. 

In this case, we can state  the main result of \cite{MR3648976} in a slightly improved manner as follows,  taking  advantage of the last statement of Theorem \ref{th:main}. (See Remark \ref{rem:previous} below.) 
\begin{Theorem}[Faure-Tsujii \cite{MR3648976}]\label{th:previous}
For $0\le k\le d_u$, the dynamical Fredholm determinant $d_k(s)$ in \eqref{eq:dynfd} extends to an entire function on the complex plane $\complex$. 
Let  $\varepsilon>0$ be an arbitrarily small positive number. Then the dynamical Fredholm determinant $d_0(s)$ for the case $k=0$ has infinitely many zeros in the neighborhood 
\[
\mathcal{R}_0=\{z\in \mathbb{C}\mid |\Re(s)|<\varepsilon\}
\]
of the imaginary axis and the remaining zeros are contained in   
\[
\mathcal{R}_1=\{s\in \mathbb{C} \mid \Re(s)<-\chi+\varepsilon\}
\]
but for finitely many exceptions, where $\chi>0$ is the hyperbolicity exponent of the flow $f^t$. (See Definition \ref{def:Af}.) The zeros of the dynamical Fredholm determinants $d_k(s)$ for $0 < k\le d$ are contained in the region $\mathcal{R}_1$ but for finitely many exceptions. 

Consequently, from the relation \eqref{eq:alt_prod}, the zeros of the semi-classical zeta function $\zeta_{sc}(s)$ is contained in the union $\mathcal{R}_0\cup \mathcal{R}_1$ while its poles are contained in $\mathcal{R}_1$, respectively with finitely many exceptions. Further $\mathcal{R}_0$ does contain infinitely many zeros of $\zeta_{sc}(s)$. 
\end{Theorem}
\begin{Remark} \label{rem:previous}
In \cite{MR3648976}, we proved a very similar statement but we only showed that $d_k(s)$ are meromorphic. This improvement is brought by the ``fiber-wise cohomology argument" that we will develop in Section \ref{sec:fiber}.  
\end{Remark}

If we compare the theorem above with Theorem \ref{th:main}, we see that the zeros in the neighborhood $\mathcal{R}_0$ of the imaginary axis are the discrete eigenvalues of the generator $\mathbb{A}_0$ of the transfer operator \eqref{eq:bbLt} acting on the bottom cohomology space $\mathbb{H}^0$, which consists of distributional section of the line bundle $L$ that are constant along the unstable leafs with respect to the dynamical connection. (We can find a similar observation in the recent paper \cite{FaureGuillarmou}.) This observation is interesting because such restriction of the operator \eqref{eq:bbLt} is closely related to the geometric quantization of the flow $f^t$ in the sense of Kostant and Souriau when we take the unstable foliation as ``polarization".  In this context, it will be also of worth noting that the density of the zeros in the region $\mathcal{R}_0$ satisfies an analogue of Weyl law \cite{MR3648976}. 

\begin{Remark}\label{rem:conj}
As we have already mentioned in \cite{MR3648976} (in a slightly different expression), we expect that the semi-group of transfer operators  \eqref{eq:bbLt} acting on the higher cohomology space $\mathbb{H}^k$ ($k>0$) will be of finite rank up to a strongly contracting remainder and will not produce many singularities of the zeta function  $\zeta_{sc}(s)$. In a sense, the main result of this paper is the first step to attack this problem, which enables us to state the problem rigorously. 
\end{Remark}

\section{Extensions}\label{sec:ext}
As we have mentioned in the introduction, the main difficulty in the proof of Theorem \ref{th:main} is that the transfer operators $\cL^t_k$ are not smooth. The idea to resolve (or avoid) this difficulty  is to consider appropriate extensions of the flow $f^t$ and to embed the non-smooth operators $\cL^t_k$ in smooth transfer operators for the extended flows. Below we explain such constructions. 

\subsection{The Grassmann extensions} \label{ss:gext}
As a preliminary to the construction given in the following subsections, we discuss in a little more detail about the Grassmann extension \eqref{cd:ftg2} of the Anosov flow $f^t:M\to M$ considered in Subsection \ref{ss:scz}. 
Recall that $\pi_G:G\to M$ is the Grassmann bundle of $M$ that consists of $d_u$-dimensional subspaces of $TM$ and that $\hf^t_G:G\to G$ in \eqref{eq:fG} is the flow on $G$ induced naturally by $f^t$. 
The image $\att_G=\iota_u(M)$ of the section $\iota_u$ in \eqref{eq:iotau} is a hyperbolic attractor for the action of $\hf^t_G$, because the flow $\hf^t_G$ is exponentially contracting in the fibers in its neighborhood. We therefore can take a small tubular neighborhood $\cU_G$ in its the attracting basin so that 
\[
\hf^t_G(\cU_G)\Subset \cU_G\;\;\text{ for $t>0$\quad and} \quad\cap_{t>0}\hf^t_G(\cU_G)=\att_G.
\] 
Henceforth we write $V_G$ for the restriction of the tautological vector bundle to $\cU_G$, that is, we set
\begin{equation}\label{eq:tautologicalbundle}
\pi_V:V_G\to \cU_G, \quad V_G:=\{(q,v)\in \cU_G\times TM\mid  v\in [q]\},
\end{equation}
where $[q]$ denotes the $d_u$-dimensional subspace of $T_{\pi_G(q)}M$ represented by $q$. 
The flow $f^t$ for $t\ge 0$ naturally induces the push-forward action on $V_{G}$:
\[
(Df^t)_*:V_{G}\to V_{G}, \quad (Df^t)_*(q,\sigma)=(\hf^t_G(q), (Df^{t})(\sigma)).
\]
Recall that, in Subsection \ref{ss:scz},  we introduced a $C^\infty$ one-parameter group of vector bundle maps $\hg^t:\hL\to \hL$ of a $C^\infty$ line bundle $\hL$ on $G$ over $\hf^t_G$. The line bundle $L$ is defined as the pull-back of $\hL$ by  $\iota_u$ and then $\hg^t$ induces a continuous one-parameter group of vector bundle maps $g^t:L\to L$ over $f^t$. 
We write $\hL^*$ and $L^*$ for the dual line bundle of $\hL$ and $L$ respectively, so that $L^*$ is identified with the pull-back of $\hL^*$ by $\iota_u$. 

In addition, we introduce two line bundles on $G$:
\[
V_G^{\wedge d_u}\quad\text{and}\quad |\Det^*_G|=\pi_{G}^*|\Det^*|.
\]
The former is the determinant line bundle of $V_G$ and the latter is the pull-back of the $1$-density bundle $|\Det^*|$ of $T^*M$ by  $\pi_G:G\to M$. The pull-backs of these two line bundles by $\iota_u$ are by definition the determinant bundle $E_u^{\wedge d_u}$ of the unstable subbundle $E_u$ and the $1$-density bundle $|\Det^*|$ of the cotangent bundle $T^*M$. 
In parallel to  the argument for the line bundle $L$ in Subsection \ref{ss:extder}, we see the line bundles $(E_u)^{\wedge d_u}$ and $T^*M^{\wedge d}$ are equipped with the unique flat connections along the unstable leafs that are invariant with respect to the natural action of the flow $f^t$ on them. Such flat connections are called the dynamical flat connections as well. 

From the line bundles $\hL^*$, $|\Det_G^*|$ and $V_G^{\wedge d_u}$, we define the line bundle
\begin{align*}
&\pi_{\frL_{\otimes}}:\frL_{\otimes}\to G,\qquad \frL_{\otimes}=\hL^*\otimes |\Det_G^*| \otimes V_G^{\wedge d_u}
\intertext{by taking the tensor product and also the $3$-dimensional vector bundle}
&\pi_{\frL_{\oplus}}:\frL_{\oplus}\to G,\qquad \frL_{\oplus}=\hL^*\oplus |\Det^*_G| \oplus V_G^{\wedge d_u}
\end{align*}
by taking the direct (or Whitney) sum. 
Clearly flow $f^t$ and the one-parameter group $\hg^t$ induce the semi-group of vector bundle maps 
\[
\frg^t_{\oplus}:\frL_{\oplus}\to \frL_{\oplus}\quad \text{and}\quad  
\frg^t_{\otimes}:\frL_{\otimes}\to \frL_{\otimes}\quad \text{for }t\ge 0
\]
over the semi-flow $\hf^t_G:G\to G$.

\subsection{Jet spaces}
Below we consider further extensions of the Anosov flow $f^t$. 
We consider the jet spaces of pairs of a $C^\infty$ hypersurface $S$ of dimension $d_u$ in $M$
and a $C^\infty$ flat connection of the line bundles $\frL_{\oplus}$ along the lift of the hypersurface $S$ to $G$. The precise definition is given as follows.

First we consider a germ of  $d_u$-dimensional $C^\infty$  hypersurface in $M$ which is represented by a germ of $C^\infty$ immersion 
\[
\bs:(\real^{d_u},0)\to (M,\bs(0)).
\]
Its differential gives a germ of hypersurface in $G$ represented by 
\[
\bs_G:(\real^d,0)\to (G,D\bs_G(0)), \quad \bs_G(x)=(\bs(x), [(D\bs)_x(\real^{d_u})]).
\]
For convenience, we consider only the germ $\bs$ such that the image $\bs_G(0)$ of the origin $0$ is contained in the neighborhood $\cU_G$ of the attractor $\att_G$. 

Second we consider a triple of germs of $C^\infty$ flat connections of the line bundles  $\hL^*$, $|\Det_G^*|$ and $V_G^{\wedge d}$ along the lift $\bs_G$ of $\bs$. We suppose that such a triple is represented by a germ of $C^\infty$ map 
\[
\bh=(h_0,h_1,h_2):(\real^{d_u},0)\to \frL_{\oplus}=\hL^*\oplus |\Det^*_G| \oplus V_G^{\wedge d_u}
\]
satisfying $\pi_{\frL_{\oplus}}\circ \bh=\bs_G$ and $h_i(0)\neq 0$ for $i=0,1,2$. (Each $h_i$ determines a unique flat connection with respect to which $h_i$ is a constant section.)

Let $\cJ$ be the set of pairs $\bsigma=(\bs,\bh)$ of germs of $C^\infty$ immersions 
\[
\bs:(\real^{d_u},0)\to (M,\bs(0))\quad \text{and}\quad  \bh:(\real^{d_u},0)\to \frL_{\oplus}
\]
satisfying the conditions as above. For each integer $\ell\ge 1$, we consider the equivalence relation $\sim_\ell$ on $\cJ$ such that $(\bs,\bh)\sim_\ell (\tilde{\bs},\tilde{\bh})$ if and only if there exist a triple $\mathbf{c}=(c_0,c_1,c_2)$ of non-zero constants and  a germ $\varphi:(\real^{d_u},0)\to (\real^{d_u},0)$ of $C^\infty$ diffeomorphism such that 
\begin{enumerate}
\item $\tilde{\bs}$ and $\bs\circ \varphi$ have contact of order $\ell+1$ at $0$, and 
\item $\tilde{\bh}$ and $(\mathbf{c}\cdot \bh)\circ \varphi$ have contact of order $\ell$ at $0$ 
\end{enumerate}
where $\mathbf{c}\cdot \bh$ denotes the component-wise multiplication. 
The quotient space $\bbJ^\ell:=\cJ/\sim_\ell$ is naturally equipped with the structure of smooth fiber bundle over $M$, which is denoted by
\[
\pi:\bbJ^\ell\to M, \qquad \pi([(\bs,\bh)]_{\ell})=\bs(0)
\]
where $[(\bs,\bh)]_\ell$ is the equivalence class of $(\bs,\bh)$ with respect to the equivalence relation $\sim_{\ell}$. For convenience, we set $\bbJ^0:=\cU_G\subset G$ so that we have the sequence of natural projections:
\begin{equation}\label{eq:Jseq}
\begin{CD}
 \bbJ^{\ell+1} @>>> \bbJ^{\ell}  @>>> \cdots @>>> \bbJ^{1} @>>> \bbJ^0=\cU_G@>>> M \end{CD}
\end{equation}
For integers $i\ge j\ge 0$, we write $\pi_{i,j}:\bbJ^i\to \bbJ^j$ and $\pi_i:\bbJ^i\to M$ for the projections obtained as compositions of those in this sequence. 

The semi-group $\frg^t_\oplus$ acts on $\cJ$ in such a way that   
\[
\mathcal{G}^t:\cJ\to \cJ,\quad
\mathcal{G}^t(\bsigma):=(f^t\circ \bs, \frg^t_{\oplus}\circ \bh)\quad \text{for }\bsigma=(\bs,\bh).
\]
This induces the (well-defined) one-parameter semi-group 
\[
\hf^t_\ell:\bbJ^\ell\to \bbJ^\ell,\qquad [\bsigma]_\ell \mapsto [\mathcal{G}^t(\bsigma)]_\ell.
\]
Since the action of $\hf^t_{\ell}$ contracts the fibers of the projection $\pi_{\ell,0}:\bbJ^\ell\to \bbJ^0=\cU_G$, it admits a unique continuous section 
\[
\iota_{u}^\ell:M\to \bbJ^\ell\quad \text{such that $\pi_{\ell, 0}\circ \iota_{u}^\ell=\iota_u$}
\]
such that the following diagram commutes:
\begin{equation}\label{cd:ftg}
\begin{CD}
\bbJ^\ell @>{\hf^t_\ell}>> \bbJ^\ell\\
@A{\iota_u^\ell}AA @A{\iota_u^\ell}AA\\
M@>{f^t}>> M
\end{CD}
\end{equation} 
The image  $\mathcal{A}_\ell:=\iota_u^\ell(M)$ of the section $\iota_u^\ell:M\to \bbJ^\ell$ is a hyperbolic attractor for the semi-flow $\hf^t_\ell:\bbJ^\ell\to \bbJ^\ell$ and we can take a tubular neighborhood $\mathcal{U}_{\ell}$ of $\mathcal{A}_\ell$ so that $\pi_{\ell,\ell-1}(\mathcal{U}_{\ell})=\mathcal{U}_{\ell-1}$ and that
\[
\hf^t_\ell(\mathcal{U}_\ell)\Subset \mathcal{U}_\ell\quad\text{ for $t>0$\quad and }\quad\cap_{t>0} \hf^t_\ell(\mathcal{U}_\ell)=\mathcal{A}_\ell.
\] 
That is to say, the situation is similar to the case of the Grassmann extension considered in the last subsection though that dimensions of the fibers will be much larger.

\subsection{Transfer operators for the extended semi-flows} 
Below we set up a few vector  bundles on the region $\cU_\ell\subset \bbJ^\ell$ and introduce semi-groups of transfer operators acting on the sections of them. 
Recall that $\pi_\ell:\cU_\ell\to M$ is the restriction of the natural projection from $\bbJ^\ell$ to $M$. The kernel of its differential $D\pi_\ell:T\cU_\ell\to TM$,  
\[
\ker D\pi_\ell=\{(x,v)\in T\cU_\ell \mid D\pi_\ell(v)=0\},
\]
is a $C^\infty$ vector subbundle of $T\cU_\ell$ with dimension
\[
d_\ell:=\dim \ker D\pi_\ell=\dim \bbJ^\ell-\dim M.
\] 
The  $1$-density line bundle of the dual of $\ker D\pi_\ell$ is written
\[
|\Omega_\ell|:=|((\ker D\pi_\ell)^*)^{\wedge d_\ell}|.
\]
For integers $\ell\ge 0$ and $0\le k\le d_u$, we consider the vector bundles
\begin{align*}
\hat{V}_{k,\ell}&:=\pi^*_{\ell,0}(\hL)\otimes  \pi^*_{\ell,0}(V_G^*)^{\wedge k}\otimes |\Omega_\ell|
\intertext{
and}
\hat{V}_{k,\ell}^\dag&:=\pi^*_{\ell,0}(\frL_{\otimes})\otimes  
\pi^*_{\ell,0}(V_G^*)^{\wedge k}
\end{align*}
 over $\cU_\ell$, where $\pi^*_{\ell,0}$ denotes the pull-back of vector bundles  by the projection $\pi_{\ell,0}:\bbJ^\ell\to \bbJ^0=\cU_\ell$. 
\begin{Remark}
As we will see below, we regard the space of sections of $\hat{V}_{k,\ell}^\dag$ as the dual of the space of sections of $\hat{V}_{k,\ell}$. 
\end{Remark}
 
The  flow $f^t$ together with the semi-group $\hat{g}^t$ naturally induces the push-forward action  on these vector bundles over $\hf^t_\ell$, which are denoted  by 
\begin{equation}\label{eq:Ftkl}
\hat{F}^t_{k,\ell}:\hat{V}_{k,\ell}\to \hat{V}_{k,\ell},\qquad
\hat{F}^t_{\dag,k,\ell}:\hat{V}_{k,\ell}^\dag\to \hat{V}_{k,\ell}^\dag\qquad\text{respectively}.
\end{equation}
Recall the vector bundle $V_k$ on $M$ in \eqref{eq:Vk} and set 
\[
V_k^{\dag}=(L^*\otimes E_u^{\wedge d_u}\otimes |\Det^*|) \otimes (E_u^*)^{\wedge k} \quad \text{for $0\le k\le d$}. 
\]
Clearly this vector bundle is identified with the pull-back of the vector bundle $\hat{V}^\dag_{k,\ell}$ by the section $\iota_u^\ell:M\to \bbJ^\ell$ for any $\ell \ge 1$. Thus, if we write
\[
F^t_{k}:V_k\to V_k\quad \text{and}\quad 
F^t_{\dag,k}:V_k^\dag\to V_k^\dag
\]
for the continuous semi-group of vector bundle maps induced by the flow $f^t$ and the semi-group $g^t$, 
we have the commutative diagrams 
\begin{equation}\label{cd:Fi}
\begin{CD}
\hat{V}_{k,\ell}@>{\hat{F}^t_{k,\ell}}>>\hat{V}_{k,\ell}\\
@A{\iota_u^\ell}AA@A{\iota_u^\ell}AA\\
V_k@>{F^t_{k}}>>V_k
\end{CD}
\qquad\qquad 
\begin{CD}
\hat{V}_{k,\ell}^\dag@>{\hat{F}^t_{\dag,k,\ell}}>>\hat{V}_{k,\ell}^\dag\\
@A{\iota_u^\ell}AA@A{\iota_u^\ell}AA\\
V_k^\dag@>{F^t_{\dag,k}}>>V_k^\dag
\end{CD}
\end{equation}

Below we consider the transfer operators induced by the vector bundle maps $\hat{F}^t_{k,\ell}$ and $F^t_{k}$. But, since we would like to consider the action of such transfer operators acting on the distributional sections, we start with considering their duals. 
The semi-group $\hat{F}^t_{\dag,k,\ell}$ of vector bundle maps  in  \eqref{eq:Ftkl} induces the semi-group of the \emph{pull-back} operators 
\begin{equation}\label{eq:Fkl}
(\hat{F}^t_{\dag,k,\ell})^*:\Gamma^{\infty}(\hat{V}_{k,\ell}^\dag)\to \Gamma^{\infty}(\hat{V}^\dag_{k,\ell}), \qquad (\hat{F}^{t}_{\dag,k,\ell})^*u(x)= \hat{F}^{-t}_{\dag,k,\ell}(u(\hf_\ell^t(x))).
\end{equation}
The the pull-back operators by $F^t_{\dag,k}$ and $\iota_u^\ell$ are denoted respectively by
\begin{equation}\label{eq:ftiu}
(F^t_{\dag,k})^*:\Gamma^{\infty}_{cu}(V_k^\dag)\to \Gamma^{\infty}_{cu}(V_k^\dag)\quad \text{and}\quad (\iota_u^\ell)^*:\Gamma^\infty_{0}(\hat{V}_{k,\ell}^\dag)\to \Gamma^{\infty}_{cu}(V_k^\dag).
\end{equation}
Then correspondingly to 
the commutative diagram on the right-hand side of \eqref{cd:Fi}, we have the commutative diagram of transfer operators
 \begin{equation}\label{cd:gamma0}
\begin{CD}
 {\Gamma}^{\infty}(\hat{V}^\dag_{k,\ell})@>{(\hat{F}^t_{\dag,k,\ell})^*}>> {\Gamma}^{\infty}(\hat{V}^\dag_{k,\ell})\\
@V{(\iota_u^\ell)^*}VV@V{(\iota_u^\ell)^*}VV\\
\Gamma^{\infty}_{cu}(V_k^\dag)@>{(F_{\dag,k}^t)^*}>>\Gamma^{\infty}_{cu}(V_k^\dag)
\end{CD}
\end{equation}
Below we consider the dual of this commutative diagram. (But we will consider the diagram \eqref{cd:gamma0} with $k$ replaced by $d-k$ for convenience.)  Since we have the contraction 
\[
(\cdot, \cdot): \hat{V}^\dag_{d-k,\ell}\times \hat{V}_{k,\ell}\to |\Omega_\ell|\otimes \pi^*_{\ell} |\Det^*|\simeq |\Det_{\bbJ^\ell}^*|
\]
where $|\Det_{\bbJ^\ell}^*|$ denotes the $1$-density bundle of $T^*\bbJ^\ell$, we can consider the pairing
\begin{equation}\label{eq:pairing}
\langle \cdot,\cdot\rangle: {\Gamma}^{\infty}(\hat{V}^\dag_{d-k,\ell})\times \Gamma_{0}^{\infty}(\hat{V}_{k,\ell}) \to \complex, \quad 
\langle v,u\rangle=\int (v,u)  
\end{equation}
where $\Gamma_{0}^{\infty}(\hat{V}_{k,\ell})$ denotes the space of $C^\infty$ sections of the $C^\infty$ vector bundle $\hat{V}_{k,\ell}$ with compact support. 
With this pairing, we regard  the dual space of ${\Gamma}^{\infty}(\hat{V}^\dag_{d-k,\ell})$ as the space of distributional sections of $\hat{V}_{k,\ell}$, and  write $\Gamma^{-\infty}(\hat{V}_{k,\ell})$ for it. Similarly, from the contraction
\[
(\cdot, \cdot): V_{d-k}^\dag\times V_k\to |\Det^*|
\]
we define the pairing 
\begin{equation}\label{eq:pairing2}
\langle \cdot,\cdot\rangle: {\Gamma}^{\infty}_{cu}({V}^\dag_{d-k})\times \Gamma^{\infty}_{cu}({V}_{k}) \to \complex, \quad 
\langle v,u\rangle=\int (v,u).
\end{equation}
Then we may regard the dual space of $\Gamma^\infty_{cu}(V_{d-k}^\dag)$ as the space of distributional sections of $V_{k}$, and write $\Gamma^{-\infty}(V_{k})$ for it. 

With the notation introduced above, we write the dual of the operator $(\hat{F}^t_{\dag, d-k,\ell})^*$ as  
\begin{equation}\label{eq:bLct}
\hat{\cL}^t_{k,\ell}:
\Gamma^{-\infty}(\hat{V}_{k,\ell})
\to 
\Gamma^{-\infty}(\hat{V}_{k,\ell}).
\end{equation}
This is an extension of the transfer operator defined as the push-forward operator by $F^t_{k,\ell}:\hat{V}_{k,\ell}\to  \hat{V}_{k,\ell}$. 
The dual of the operator $(F_k^t)^*$ is nothing but the  the transfer operator  $\cL_k^t:\Gamma^\infty(V_{k})\to \Gamma^\infty(V_{k})$ introduced in \eqref{eq:cLtk}. Hence, if we write  
\begin{equation}\label{eq:iotau}
(\iota_u^{\ell})_*:\Gamma^\infty_{cu}(V_k)\to \Gamma^{-\infty}(\hat{V}_{k,\ell})
\end{equation}
for the restriction of the dual of $(\iota_u^{\ell})^*: {\Gamma}^{\infty}(\hat{V}^\dag_{k,\ell})\to \Gamma^{\infty}_{cu}(V_k^\dag)$, we obtain the commutative diagram
\begin{equation}\label{cd:gamma0dual}
\begin{CD}
 {\Gamma}^{-\infty}(\hat{V}_{k, \ell})@>{\hat{\cL}^t_{k,\ell}}>> {\Gamma}^{-\infty}(\hat{V}_{k,\ell})\\
@A{(\iota_u^\ell)_*}AA@A{(\iota_u^\ell)_*}AA\\
\Gamma_{cu}^{\infty}(V_{k})@>{\cL^t_{k}}>>\Gamma_{cu}^{\infty}(V_{k})
\end{CD}
\end{equation}
as the dual of \eqref{cd:gamma0}. 
 
\subsection{The exterior derivative operator along unstable leafs}
The commutative diagram \eqref{cd:gamma0dual} implies that  the non-smooth transfer operator $\cL^t_{k}$ that we are interested in is embedded in the smooth transfer operator $\hat{\cL}^t_{k,\ell}$ on the extended space by $(\iota_u^\ell)_*$ in a sense. Next we discuss about  the exterior derivative operators along unstable leafs. To proceed, let us set 
\begin{equation}\label{eq:VV}
\hat{V}_{k}=\hat{V}_{k,d-k} \quad\text{and}\quad \hat{V}^\dag_{k}=\hat{V}^\dag_{k,k}\qquad\text{for $0\le k\le d_u$}
\end{equation}
for simplicity. 
For  $0\le k<d_u$, we  are going define the operator
\begin{equation}\label{opD}
\hder_{k}^u: \Gamma^{-\infty}_{0}(\hat{V}_{k})\to \Gamma^{-\infty}_{0}(\hat{V}_{k+1})
\end{equation} 
so that the following diagram commutes:
\begin{equation}\label{cd:extDdual}
\begin{CD}
 {\Gamma}^{-\infty}(\hat{V}_{k})@>{\hder_{k}^u}>> {\Gamma}^{-\infty}(\hat{V}_{k+1})\\
@A{(\iota_u^\ell)_*}AA@A{(\iota_u^\ell)_*}AA\\
\Gamma_{cu}^{\infty}(V_{k})@>{\der_{k}^u}>>\Gamma_{cu}^{\infty}(V_{k+1})
\end{CD}
\end{equation}Similarly to the argument in the last subsection, we first introduce the operators  
\begin{equation}\label{eq:hderd}
\hder_{\dag,k}^u:\Gamma^{\infty}(\hat{V}_{k}^\dag)\to \Gamma^{\infty}(\hat{V}_{k+1}^\dag)\quad \text{for $0\le k<d_u$}
\end{equation}
and then define $\hder_k^u$ as their duals. 

Suppose that $\varphi \in  \Gamma^{\infty}(\hat{V}^\dag_{k})$ and $[\bsigma]_{k+1}\in \mathcal{U}_{k+1}\subset \bbJ^{k+1}$ are given, where $[\bsigma]_{k+1}$ denotes the equivalence class of an element $\bsigma\in \cJ$.  Below we explain how we define its image $\hder_{\dag,k}^u \varphi(\bsigma)$. Recall that $\varphi \in  \Gamma^{\infty}(\hat{V}^\dag_{k})$ is a section of the vector bundle $\hat{V}_k^\dag=\hat{V}^\dag_{k,k}$ defined on the tubular neighborhood $\mathcal{U}_{k}\subset \bbJ^k$ of the attractor $\mathcal{A}_k$.
\vspace{3mm}

\paragraph{\textbf{First Step:}}
By definition,  $\bsigma=(\bs,\bh)$ is a pair of $C^\infty$ germs  
\[
\bs:(\real^{d_u},0)\to (M,\bs(0))\quad \text{and}\quad \bh:(\real^{d_u},0)\to \frL_\oplus.
\] 
Let us write $S$ for the germ of $d_u$-dimensional hypersurface in $M$ at $\bs(0)$ given as the image of $\bs$. The component $\bh$ in $\bsigma=(\bs,\bh)$ gives a germ of $C^\infty$ flat connection of $\frL_{\otimes}$ along the its lift $\bs_G$, as it gives a germ of $C^\infty$ flat connections of the factors $\hL^*$, $|\Det_G^*|$ and $\otimes V_G^{\wedge d_u}$ of $\frL_{\otimes}$  along the lift $\bs_G$. 
Therefore the pair $\bsigma=(\bs,\bh)$ gives a germ of $C^\infty$ map  
\[
\Psi_{\bsigma}: (\real^{d_u}, 0) \to \bbJ^{k}
\]
which assign  each point $x_0\in \real^d$ in a neighborhood of the origin $0$ the element $[(\bs_{x_0},\bh_{x_0})]_{k}\in \bbJ^k$ that consists of pair of 
\begin{itemize}
\item the $(k+1)$-jet of the germ of the map $x\mapsto \bs(x+x_0)$ at $x=0$ and
\item the $k$-jet of the germ  of  the map $x\mapsto \bh(x+x_0)$ at $x=0$.
\end{itemize}
Note that the first jet of the germ $\Psi_{\bsigma}$ is determined by the element $\bsigma\in \bbJ^{k+1}$. 

\vspace{3mm}

\paragraph{\textbf{Second Step:}}
Next we pull-back the section $\varphi: \mathcal{U}_k \to \hat{V}^\dag_{k}$ by $\Psi_{\bsigma}$ and obtain the section
\[
\Psi_{\bsigma}^* \varphi:(\real^{d_u},0)\to  \bs^*_G(\frL_{\otimes})\otimes  (T^*\real^{d_u})^{\wedge k}.
\] 
Here we identify the pull-back of the vector bundle $\hat{V}^\dag_{k}$ by $\Psi_{\bsigma}$ with
\[
\bs^*_G(\frL_{\otimes}\otimes (V_G^*)^{\wedge k})=
\bs^*_G(\frL_{\otimes})\otimes D\bs^*((TS)^*)^{\wedge k}) =
\bs^*_G(\frL_{\otimes})\otimes  (T^*\real^{d_u})^{\wedge k}.
\]
Note that the first jet of the section $\Psi_{\bsigma}^* \varphi$ thus defined is determined by $\bsigma$. 
\vspace{3mm}

\paragraph{\textbf{Third Step:}}
The component $\bh$ in $\bsigma=(\bs,\bh)$ induces a flat connection of the line bundles $\bs^*_G(\frL_{\otimes})$ and therefore we can define the exterior derivative of $\Psi_{\bsigma}^* \varphi$, trivializing $\bs^*_G(\frL_{\otimes})$ by that flat connection.  We write  $\der(\Psi_{\bsigma}^* \varphi)$ for the result of such operation. 
Since the first jet of $\Psi_{\bsigma}^* \varphi$ at the origin $0$ is determined by $\bsigma$, so is the value of $\der (\Psi_{\bsigma}^* \varphi)$ at $0$. Finally, as we did in the last step, we identify $\der(\Psi_{\bsigma}^* \varphi)(0)\in \bs^*_G(\frL_{\otimes})\otimes  (T^*\real^{d_u})^{\wedge (k+1)}$ with an element in $\hat{V}_{k+1}^\dag$ at $\bsigma\in \mathcal{U}_{k+1}$, for which we write $\hder^u_{\dag,k}\varphi(\bsigma)$. The element $\hder^u_{\dag,k}\varphi(\bsigma)$ depends on $\bsigma$ in $C^\infty$ manner and therefore we obtain a smooth section 
$\hder^u_{\dag,k}\varphi\in \Gamma^{\infty}(\hat{V}_{k+1}^\dag)$. 

Once we define the operator $\hder_{\dag,k}^u:\Gamma^{\infty}(\hat{V}_{k}^\dag)\to \Gamma^{\infty}(\hat{V}_{k+1}^\dag)$  as above, we define the operator $\hder_k^u$ in \eqref{opD} as 
\[
\hder_{k}^u=(-1)^{d_u-k}\cdot (\hder_{\dag,k}^u)^*: \Gamma^{-\infty}_{0}(\hat{V}_{k})\to \Gamma^{-\infty}_{0}(\hat{V}_{k+1})
\]
where  $(\hder_{\dag,k}^u)^*$ denotes the dual of $\hder_{\dag,k}^u$ with respect to the pairing \eqref{eq:pairing} and we put the sign $(-1)^{d_u-k}$ for convenience in the argument below. 

Trivializing  the line bundle $L^*\otimes |\Det^*|\otimes (E^u)^{\wedge d_u}$ with respect to the dynamical flat connection, we define the exterior derivative operators along the unstable foliation
\begin{equation}\label{eq:derd}
\der^u_{\dag,k}:\Gamma^\infty_{cu}({V}^\dag_k)\to \Gamma^\infty_{cu}({V}^\dag_{k+1}), \quad k=0,1,2,\cdots, d_u-1. 
\end{equation}
Then the following diagram commutes
\[
\begin{CD}
\Gamma^\infty(\hat{V}^\dag_{k})@>{\hder_{\dag,k}^u}>>\Gamma^\infty(\hat{V}^\dag_{k+1})\\
@V{(\iota_u^k)^*}VV@V{(\iota_u^k)^*}VV\\
\Gamma^\infty_{cu}(V^\dag_k)
@>{\der_{\dag,k}^u}>>
\Gamma^\infty_{cu}(V^\dag_{k+1})
\end{CD}
\]
Further the operator $\hder^u_{\dag,k}$ commutes with the operator $(F^t_{\dag,k,k})^*$ in \eqref{eq:Fkl}. Therefore, writing  $F^t_{k}$ for $F^t_{k,k}$ correspondingly to \eqref{eq:VV},  we obtain the following commutative diagram
\begin{equation}\label{cd:full}
\begin{tikzcd}
& \Gamma^\infty(\hat{V}^\dag_{k}) 
\arrow[dl,"(F^t_{k})^*"'] \arrow[rr,"\hder^u_{\dag,k}"]\arrow[dd,"(\iota_{u}^k)^*" near end]  & & \Gamma^\infty(\hat{V}^\dag_{k+1}) \arrow[dl, "(F^t_{k+1})^*"']\arrow[dd,"(\iota_{u}^{k+1})^*" near end]  \\ \Gamma^\infty(\hat{V}_{k}^\dag) \arrow[dd, "(\iota_{u}^k)^*" near end]
\arrow[rr, "\hder^u_{\dag,k}" near end, crossing over]  & & \Gamma^\infty(\hat{V}^\dag_{k+1})\\
& \Gamma^\infty_{cu}(V_k^\dag) \arrow[dl, "(F^t_{\dag,k})^*"] \arrow[rr, "\der^{u}_{\dag,k}" near start] & & \Gamma^\infty_{cu}(V_{k+1}^\dag) \arrow[dl,"(F^t_{\dag,k})^*"] \\
\Gamma^\infty_{cu}(V_k^\dag) \arrow[rr, "\der^{u}_{\dag,k}" near end] & & \Gamma^\infty_{cu}(V_{k+1}^\dag)\arrow[from=uu, "(\iota_{u}^{k+1})^*" near end, crossing over]
\end{tikzcd}
\end{equation}
 
Now we consider the dual of this diagram \eqref{cd:full}. We first check that, in view of the pairing \eqref{eq:pairing2},  the exterior derivative operator  
$\der^u_{k}$ along the unstable foliation introduced in \eqref{def:deltau} is the dual of the exterior derivative operator $\der^{\dag}_{d_u-k-1}$ in \eqref{eq:derd}
up to  multiplication by $\pm 1$. In fact, we prove
\begin{equation}\label{eq:realduv}
\langle v, \der^u_k u\rangle =(-1)^{d_u-k} \langle\der^{u}_{\dag,d_u-k-1} v, u\rangle
\end{equation}
for any $v\in \Gamma^\infty_{cu}(V^\dag_{d_u-k-1})$ and $u\in \Gamma^\infty_{cu}(V_k)$, where $\langle \cdot, \cdot\rangle$ denotes the pairing \eqref{eq:pairing2}.   For the proof, we may suppose that $v$ is supported in a small disk on $M$ which is a disjoint union of the  local unstable manifolds $W^u_{\loc}(p)$ because the both sides of the equality are linear with respect to  $v$. 

If we trivialize the line bundles $L=\iota_u^*(\hL)$ and $\iota_u^*(\frL_{\otimes})$ along the local unstable manifolds $W^u_{\loc}(p)$ according to the dynamical flat connection, we may suppose that $u$ and $v$ are $(d_u-k-1)$-form and $k$-form on $W^u_{\loc}(p)$ and 
\begin{align}\label{eq:stokes}
0&= \int_{\partial W^u_{\loc}(p)}v\wedge u\; = \int_{W^u_{\loc}(p)}d (v\wedge u)\\
&=\int_{W^u_{\loc}(p)}  d v\wedge  u +(-1)^{d_u-k-1} \int_{W^u_{\loc}(p)}  v \wedge du.\notag
\end{align}
Therefore, by taking integration  with respect to the local unstable manifolds $W^u_{\loc}(p)$,  we obtain the required equality \eqref{eq:realduv}. 
\begin{Remark} For the last sentence above, some explanation may be in order. For a section $\varphi$ of $|\Det^*|$, its integration $\int \varphi $ is well defined and can be computed on a local chart as 
\[
\int \varphi =\int \frac{\varphi(x)}{|dx|} |dx|
\]
where we denote the local coordinates by $x$.
If we take a transversal section $S$ to the unstable foliation and fix a non-vanishing section $\sigma_0$ (resp. $\tau_0$) of the determinant bundle of $T^*S$ (resp. $(T_SM/TS)^*$) on $S$. Then we can introduce continuous sections $\sigma$ (resp. $\tau$) of the determinant bundle of $E_u^*\oplus E_0^*=(E_u)^{\perp}$ (resp. $E_u^*$) so that it coincides with $\sigma_0$ (resp. $\tau_0$) and constant with respect to the dynamical flat connection along the unstable foliation. Then the integral $\int \varphi$ is computed as 
\begin{align*}
\int \varphi&=\int \frac{\varphi}{|\sigma\wedge \tau|} |\sigma\wedge \tau| =\int_S \left(\int_{W^u_{\loc}(s)} \frac{\varphi}{|\sigma\wedge \tau|} \frac{|\tau|}{|dx_s|} |dx_s|\right) \frac{|\sigma_0|}{|ds|}|ds|
\end{align*}
where $s$ denotes a coordinate on $S$ and $x_s$ denotes that on $W^u_{\loc}(s)$. The equation \eqref{eq:stokes} tells that, if we set $\varphi=d(v\wedge u)$,  the integral corresponding to that in the parentheses above vanishes, so that we obtain \eqref{eq:realduv}.
\end{Remark}

Finally, taking the dual of the diagram \eqref{cd:full} with $k$ replaced by $d_u-k-1$, we obtain the commutative diagram
\begin{equation}\label{cd:extreq}
\begin{tikzcd}
& \Gamma^{-\infty}({\hat{V}}_{k+1}) \arrow[from=dd, "(\iota_{u}^{d_u-k-1})_*" near start]
  & & \Gamma^{-\infty}({\hat{V}}_{k}) 
\arrow[ll,"\hder_{k}^{u}"']
  \\ \Gamma^{-\infty}({\hat{V}}_{k+1}) \arrow[ur,"\hat{\cL}^t_{k+1}"] 
  & & \Gamma^{-\infty}(\hat{V}_{k})\arrow[ur, "\hat{\cL}^t_{k}"]
  \arrow[ll, "\hder_k^u"' near start, crossing over]\\
& \Gamma^\infty_{cu}(V_{k+1})   & & \Gamma^\infty_{cu}(V_{k}) \arrow[ll, "\der^u_k"' near start]\arrow[uu,"(\iota_{u}^{d_u-k})_*" near start] \\
\Gamma^\infty_{cu}(V_{k+1}) \arrow[ur, "\cL^t_{k+1}"] \arrow[uu, "(\iota^{d_u-k-1}_{u})_*" near start ] & & \Gamma^\infty_{cu}(V_{k})\arrow[ll, "\der^u_{k}"']\arrow[ur, "\cL^t_{k}"]\arrow[uu,"(\iota^{d_u-k}_{u})_*" near start, crossing over]
\end{tikzcd}
\end{equation}
where $(\iota_u^{d_u-k})_*$ denotes the dual of the pull-back operator $(\iota_u^{d_u-k})^*$.


\section{Anisotropic Sobolev space}
In this section, we set up a scale of Hilbert spaces $\mathcal{H}^{r,s}(\hat{V}_{k})$ such that $\Gamma^{\infty}(\hat{V}_{k})\subset \mathcal{H}^{r,s}(\hat{V}_{k})\subset \Gamma^{-\infty}(\hat{V}_{k})$. These Hilbert spaces are called anisotropic Sobolev spaces because they are generalized Sobolev spaces with anisotropic weight function adapted to the hyperbolic structure of the flow. 

\subsection{A topologically trivial extension of an Anosov flow}\label{ss:aniso} 
We  consider a setting that abstracts those considered in the last section. (See Remark \ref{rem:model}.) 
Let $f^t:M\to M$ be a $C^\infty$ Anosov flow. Suppose that  $
\pi_B:B\to M$
is the  locally trivial $C^\infty$ fibration whose fibers are diffeomorphic to the unit disk $\mathbb{D}$ in $\real^{d^\perp}$ for some integer $d^\perp\ge 1$. 
Let  $\hf^t:B\to B$ be a $C^\infty$ semi-flow which makes the following diagram commutes
\[
\begin{CD}
B@>{\hat{f}^t}>>B\\
@V{\pi_B}VV @V{\pi_B}VV\\
M@>{f^t}>> M
\end{CD}
\] 
Further we suppose that $\hat{f}^t$ contracts the fibers of the fibration $B$ exponentially.  It implies existence of a unique continuous section 
\begin{equation}\label{eq:iota}
\iota:M\to B\quad \text{such that \;$\hat{f}^t\circ \iota=\iota\circ f^t$.}
\end{equation}
Its image $
\mathcal{A}:=\iota(M)$ 
 is a hyperbolic attractor of the semi-flow $\hat{f}^t$.

 The semi-flow $\hat{f}^t$ is uniformly hyperbolic. That is,  there exists a (forward) $D\hat{f}^t$-invariant decomposition 
\begin{equation}\label{eq:decompTb}
TB=\hat{E}_0\oplus \hat{E}_s\oplus \hat{E}_u 
\end{equation}
of the tangent bundle of $B$ such that 
\begin{enumerate}
\item the subbundle $\hat{E}_0$ is the smooth one-dimensional subbundle spanned by the generating vector field $\hat{v}$ of the semi-flow $\hat{f}^t$ and 
\item for some constant $C>1$, we have 
\begin{align*}
&\|D\hat{f}^t(v)\|\ge C^{-1} \exp(\chi t)\|v\|\quad \text{for $v\in \hat{E}_u$ and $t\ge 0$}
\intertext{and}
&\|D\hat{f}^t(v)\|\le C \exp(-\chi t)\|v\|\quad \text{for $v\in \hat{E}_s$ and $t\ge 0$.}
\end{align*}
\end{enumerate}
Note that the subbundles $\hat{E}_0$ and $\hat{E}_s$ are uniquely determined by the conditions  above, whereas $\hat{E}_u$ is unique only on the attractor $\mathcal{A}$ and not on the outside of it. In relation to the hyperbolic decomposition $TM=E_0\oplus E_s\oplus E_u$ for the Anosov flow $f^t$, it holds  
\[
D\pi_B(\hat{E}_0)=E_0\quad\text{and}\quad \hat{E}_s=D\pi_B^{-1}(E_s).
\]
Further we may and do assume
\[
D\pi_B(\hat{E}_u(x))=E_u\quad \text{at each point $x\in B$}.
\]
In the argument below, the action of the flow $\hf^t$ on the cotangent space is more important because we consider the action of the flow on the Fourier space (on local charts). Let us write 
\[
T^*B=\hat{E}_0^*\oplus \hat{E}_s^*\oplus \hat{E}_u^* 
\]
for the dual decomposition of \eqref{eq:decompTb}, where $\hat{E}_0^*$, $\hat{E}_s^*$ and $ \hat{E}_u^*$ denote the normal subbundles of $\hat{E}_s\oplus \hat{E}_u$, $\hat{E}_0\oplus \hat{E}_s$ and $\hat{E}_0\oplus \hat{E}_u$ respectively. This decomposition is preserved by the push-forward action $(D\hf^{-t})^*:T^*B\to T^*B$ of the flow $\hf^t$ on the cotangent bundle $T^*M$. 

Next we consider a $C^\infty$ vector bundle $\pi_{\hat{V}}: \hat{V}\to B$ and a $C^\infty$ one-parameter semi-group of vector bundle maps $\hat{F}^t:\hat{V}\to \hat{V}$  over $\hat{f}^t$, that is,  $\pi_{\hat{V}}\circ \hat{F}^t=\hf^t\circ \pi_{\hat{V}}$. Let $\hat{V}^*$ be the dual vector bundle of $\hat{V}$. We write 
\[
V=\iota^*\hat{V}\quad \text{and}\quad  V^*=\iota^*\hat{V}^*
\]
for the pull-back of $\hat{V}$ and $\hat{V}^*$ by $\iota:M\to B$ respectively and 
\begin{equation}\label{eq:iotastar}
\iota^*:\Gamma^\infty_{0}(\hat{V})\to \Gamma^\infty_{cu}(V)\quad\text{and}\quad
\iota^*:\Gamma^\infty_{0}(\hat{V}^*)\to \Gamma^\infty_{cu}(V^*)
\end{equation}
for the pull-back operators by  $\iota$.
The semi-group $\hat{F}^t$ induces that of the vector bundle maps $F^t:V\to V$ such that the following diagram commutes:
\[
\begin{CD}
\hat{V}@>{\hat{F}^t}>>\hat{V}\\
@VV{\iota^*}V@VV{\iota^*}V\\
V@>{F^t}>> V
\end{CD}
\]

Let $|\Omega|$ be the $1$-density line bundle  of the dual of the vector bundle
\begin{equation}\label{eq:Tperp}
\ker D\pi_B=\{ (b,v)\in TB\mid (D\pi_B)_b(v)=0\}
\end{equation}
and put
\[
\hat{W}=\hat{V}\otimes |\Omega|, \quad
\hat{W}^*=\hat{V}^*\otimes \pi_B^*|\Det^*|,\quad
W=V,\quad W^*=V^*\otimes |\Det^*|.
\]
Then we consider the pairings
 \begin{align*}
& \langle \cdot, \cdot\rangle : \Gamma^\infty_{0}(\hat{W})\times \Gamma^\infty(\hat{W}^*), \quad \langle u,v\rangle =\int_B u v\\
& \langle \cdot, \cdot\rangle : \Gamma^\infty_{cu}(W)\times \Gamma^\infty_{cu}(W^*), \quad \langle u,v\rangle =\int_M u v.
 \end{align*}
With these pairing, we regard the dual spaces of $\Gamma^\infty(\hat{W}^*)$ and $\Gamma^\infty_{cu}({W}^*)$ as the spaces of distributional sections of $\hat{W}$ and $W$ and write  
$\Gamma^{-\infty}(\hat{W})$ and $\Gamma^{-\infty}(W^*)$ for them respectively. The dual of $\iota^*$ in \eqref{eq:iotastar} is written
\begin{equation}\label{eq:iota_star}
\iota_*:\Gamma_{0}^{-\infty}(W)\to \Gamma_{0}^{-\infty}(\hat{W}).
\end{equation}
The semi-group $\hat{F}^t$ and the semi-flow $\hat{f}^t$ naturally induces the push-forward action 
\[
\hat{F}^t_W:=\hat{F}^t\otimes (D\hf^{-t})^*:\hat{W}\to \hat{W}.
\]
The semi-group of transfer operators corresponding to $\hat{F}^t_W:\hat{W}\to \hat{W}$ and $F^t:V\to V$ are denoted respectively by
\begin{align}\label{eq:hLtc}
&\hat{L}^t:\Gamma^\infty_{0}(\hat{W})\to \Gamma^\infty_{0}(\hat{W}), \quad \hat{L}^t u(b)=\hat{F}^t_W(u(\hat{f}^{-t}(b)))
\intertext{and}
&L^t:\Gamma^\infty_{cu}(W)\to \Gamma^\infty_{cu}(W) , \quad {L}^t u(x)={F}^t(u(f^{-t}(x))).
\label{eq:Ltc}
\end{align}
By using duality with respect to the pairing introduced above, we may and do extend the operator \eqref{eq:hLtc}  to the continuous operator 
\begin{equation}\label{eq:hLt_ext}
\hat{L}^t:\Gamma^{-\infty}_{0}(\hat{W})\to \Gamma^{-\infty}_{0}(\hat{W}).
\end{equation}
Then we have the commutative diagram 
\begin{equation}\label{cd:gwc}
\begin{CD}
\Gamma^{-\infty}_{0}(\hat{W})@>{\hat{L}^t}>>\Gamma^{-\infty}_{0}(\hat{W})\\
@A{\iota_*}AA @A{\iota_*}AA\\
\Gamma^{\infty}_{cu}(W)@>{L^t}>>\Gamma^{\infty}_{cu}(W)
\end{CD}
\end{equation}
\begin{Remark}\label{rem:model}
The semi-flow $\hat{f}^t$ is a model of $f^t_k:\mathcal{U}\to \mathcal{U}$ in the last section.
The vector bundles $W$ and $|\Omega|$ correspond to $\pi^*_{\ell,0}\hL\otimes \pi^*_{\ell,0}(V^*_G)^{\wedge k}$ and $|\Omega_\ell|$ respectively.  Then the other correspondences should be clear.  
\end{Remark}

\subsection{The definition of the anisotropic Sobolev space} \label{ss:def_aniso}
We now gives the definition of the anisotropic Sobolev space. For more details, we refer \cite{FaureSjostrand11}.
Let $\mathbf{P}(T^*B)$ is the projectivized cotangent bundle of $B$.  
First we take and fix a $C^\infty$ function 
\[
\mu:\mathbf{P}(T^*B)\to [s,r]
\]
such that, for some constant $K>1$, 
\[
\mu([(b,\xi)])=
\begin{cases}
0& \text{if $\max\{\|\xi_0\|,\|\xi_u\|\}\le \|\xi_s\|/K$};\\
1/2& \text{if $\max\{\|\xi_s\|,\|\xi_u\|\}\le \|\xi_0\|/K$};\\
1& \text{if $\max\{\|\xi_0\|,\|\xi_s\|\}\le \|\xi_u\|/K$};\\
\end{cases}
\]
where $\xi=\xi_0+\xi_s+\xi_u$ with $\xi_0\in \hat{E}_0^*(b)$, $\xi_s\in \hat{E}_s^*(b)$ and $\xi_u\in \hat{E}_u^*(b)$ and that
\begin{equation}\label{eq:wrsdec}
\mu([(D\hf^{-t})^*(b,\xi)])\ge \mu([(b,\xi)])\quad \text{for $t\ge 0$}. 
\end{equation}
For $s<0<r$,  we define 
\[
\weight^{s,r}:T^*B\to \real_+,\quad \weight^{s,r}(b,\xi)=
\left\langle \|\xi\|/\langle \xi_0\rangle \right\rangle^{r+(s-r)\mu([(b,\xi)])} .
\]
Then, from \eqref{eq:wrsdec} and  hyperbolic property of the flow $f^t$, it follows
\begin{align}
&\weight^{s,r}((D\hf^{-t})^*(b,\xi))\le C\cdot \weight^{s,r}(b,\xi)\quad\text{for $(b,\xi)\in T^*B$.}
\label{eq:wsr1}
\intertext{Further it holds}
&\weight^{s,r}((D\hf^{-t})^*(b,\xi))\le \exp(-\min\{|s|, r\}\cdot \chi t) \cdot {\weight}^{s,r}(b,\xi)
\label{eq:wsr2}
\end{align}
provided that $t>0$ and the ratio $\|\xi\|/\langle \xi_0\rangle$ is sufficiently large. 
\begin{Remark} Similarly to \eqref{eq:wsr1} but a little more precisely, we have that, for any constant $0<\theta\le 1$, 
\begin{equation}\label{eq:wsrmotive2}
\weight^{s,r}(\xi)\le \exp(-\min\{|s|\theta, r\}\cdot \chi t/2) \cdot \weight^{s,r}(b,\xi)
\end{equation}
for  $(b,\xi)\in T^*B$ satisfying $\|\xi\|\ge \exp(\theta \chi t/2)\langle|\xi_0|\rangle$, provided that $t$ is sufficiently large.  
\end{Remark}

\begin{Definition}
We define the anisotropic Sobolev space $\mathcal{H}^{s,r}(B)$ as the completion of the space $C^\infty_{0}(B)$ with respect to  the norm 
\[
\|u\|_{\mathcal{H}^{s,r}}:=\|\mathrm{Op}(\weight^{s,r})u\|_{L^2}.
\]
where $\mathrm{Op}(\weight^{s,r})$ denotes the pseudo-differential operator with the symbol $\weight^{s,r}$. 
Then we  set 
\[
\mathcal{H}^{s,r}(\hat{W})=\mathcal{H}^{s,r}(B)\otimes _{C^\infty(M)}\Gamma^\infty(\hat{W}).
\]
\end{Definition}

\subsection{A few propositions related to the anisotropic Sobolev space}
The next proposition is basically  a consequence of the property \eqref{eq:wsr2} of the weight function  $\weight^{s,r}$. 

\begin{Proposition}[\cite{FaureSjostrand11}]\label{th:standard} For any $s<0<r$, the transfer operator $\hat{L}^t$ for $t\ge 0$ restricts to a strongly continuous one-parameter semi-group of bounded operators 
\begin{equation}\label{eq:LT}
\hat{L}^t: \mathcal{H}^{s,r}(\hat{W})\to \mathcal{H}^{s,r}(\hat{W}).
\end{equation}
There exists a constant $C>0$ independent of the parameters $s$ and $r$ such that the spectral set of the  generator $A$ of $\hat{L}^t$ on the region 
\begin{equation}\label{eq:regionrs}
\Re(s)> -\min\{|s|, r\}\chi+C
\end{equation}
consists of discrete eigenvalues with finite multiplicity.  
\end{Proposition}

For the operator $\iota_*$ in \eqref{eq:iota_star}, we have
\begin{Proposition}\label{pp:diff} The operator $\iota_*$  restricts to a continuous operator  
\begin{equation}\label{eq:iota_ast}
\iota_*:\Gamma^\infty_{cu}(W)\to \mathcal{H}^{s,r}(\hat{W})\subset \Gamma^{-\infty}(\hat{W})
\end{equation}
provided that $s<0$ is sufficiently small according to $\dim B$. 
\end{Proposition}
\begin{proof} 
Let $\mu$ be a $C^\infty$ section of $W$ along an local center-unstable manifold $W^{cu}_{\loc}(p)$ with compact support. Then we define $\iota_*\mu$ as the distributional section of $\hat{W}$ defined by 
\[
\iota_*\mu(\varphi)=\int_{W^{cu}_{\loc}(p)} \langle \mu(x), \iota^*\varphi(x) \rangle\; dx \quad \text{for any $C^\infty$ section $\varphi$ of $\hat{W}^*$.}
\]
This is a distributional section of $\hat{W}$ that concentrate on the local center-unstable manifold $W^{cu}_{\loc}(p')$ of the unique point $p'\in \att$ that projects to $p$. Its wavefront set is contained in the normal bundle of the local center-unstable manifold and contained in the subbundle $E_u^*$. Hence, from the definition of the weight function $\mathcal{W}^{s,r}$, we see that any of such distributional sections belongs to $\mathcal{H}^{s,r}(\hat{W})$ provided that $s<0$ is sufficiently small according to $\dim B$. The image of \eqref{eq:iota_ast} is written as an integration of such measures (with uniformly bounded norm in $\mathcal{H}^{s,r}(\hat{W})$). Thus, recalling the remark given in Subsection \ref{ss:remwu}, we obtain the claim of the proposition. \end{proof}

\begin{Proposition}\label{pp:diffop}
Any first order differential operator $D:\Gamma^{\infty}_{0}(\hat{W})\to \Gamma^{\infty}_{0}(\hat{W})$ with $C^\infty$ coefficients  extends to a bounded operator 
\[
D:\mathcal{H}^{s,r}(\hat{W})\to \mathcal{H}^{s-1,r-1}(\hat{W})\quad\text{for $s<0<1<r$. }
\]
\end{Proposition}
\begin{proof}
This proposition follows from the composition formula of pseudo-differential operators or a straightforward estimate using local charts. \end{proof}

\subsection{Dynamical Fredholm determinant}

For a periodic orbit $\po\in \PO$ of the flow $f^t$, we write  $\hat{\po}$ for the unique periodic orbit of $\hat{f}^t$ that projects down to $\po$ by $\pi_B:B\to M$. We take a point $p$ on $\hat{\po}$ and set 
\[
\hat{V}_\po:=F^{|\po|}(p):\hat{V}(p)\to \hat{V}(p),\quad 
|\Omega_{\po}|:=(D\hat{f}^{-|\po|}(p))^*:|\Omega|(p)\to |\Omega|(p).
\]
Then the dynamical Fredholm determinant of the semi-group of operators $L^t$ in  \eqref{eq:Ltc} is  computed  as  
\begin{align}
d(s)&=\exp\left(-\sum_{\po\in \PO} \sum_{n=1}^\infty 
\frac{e^{-sn|\po|}\cdot \mathrm{Tr}\, {V}_{\po}^n}{n\cdot |\det(\mathrm{Id}-P_\po^{-n})|} \right)
\label{eq:dsabs}
\end{align}
Let $\cH^{s,r}(W)$ be the completion of the space $\Gamma^\infty_{cu}(W)$ with respect to the pull-back of the norm on $\cH^{s,r}(\hat{W})$, so that $\iota_*:\cH^{s,r}(W)\to \cH^{s,r}(\hat{W})$ is an isometric embedding and we have the following commutative diagram that extends \eqref{cd:gwc}:
\begin{equation}\label{cd:gwc2}
\begin{CD}
\cH^{s,r}(\hat{W})@>{\hat{L}^t}>>\cH^{s,r}(\hat{W})\\
@A{\iota_*}AA @A{\iota_*}AA\\
\cH^{s,r}(W)@>{L^t}>>\cH^{s,r}(W)
\end{CD}
\end{equation}
From  Proposition \ref{th:standard}, we see that the generator of $L^t:\cH^{s,r}(W)\to \cH^{s,r}(W)$ as discrete spectrum on the region \eqref{eq:regionrs}. Then we have  
\begin{Proposition}\label{pp:lf2}
The dynamical Fredholm determinant ${d}(s)$ in \eqref{eq:dsabs} extends to a holomorphic function on the complex plane $\complex$ and its zeros coincide with the discrete eigenvalues of the generator of $L^t:\cH^{s,r}(W)\to \cH^{s,r}(W)$ on the region \eqref{eq:regionrs}. 
\end{Proposition}
The proof of this proposition will be given in the last section of this paper. Note that, 
though the claim of the proposition is natural and parallel to the last proposition, it does not follow from the existing argument in the literature (to the best of the author's understanding) because $L^t:\cH^{s,r}(W)\to \cH^{s,r}(W)$ is not smooth. We prove Proposition \ref{pp:lf2} in the last section of this paper, developing a cohomological argument along the fibers of the fibration $\pi_B:B\to M$.


\subsection{Proof of Theorem \ref{th:main}}
We now give the proof the main theorem, Theorem \ref{th:main}. 
Recall the setting that we discussed in Section \ref{sec:ext} and check
\begin{enumerate}
\item  The tubular neighborhood $\mathcal{U}_{d_u-k}\subset \bbJ^{d_u-k}$ of the attractor $\mathcal{A}_{d_u-k}$ is a fibration over the manifold $M$ whose fibers are diffeomorphic to the disk of dimension $d^\perp_k=\dim \bbJ^{d_u-k}-\dim M$. 
\item $\hat{f}^t_{d_u-k}:\mathcal{U}_{d_u-k}\to \mathcal{U}_{d_u-k}$ is a $C^\infty$ semi-flow satisfying $\pi_{d_u-k}\circ \hat{f}^t_{d_u-k}=f^t\circ \pi_{d_u-k}$ and is exponentially contracting on the fibers. 
\item The vector bundle  $\hat{V}_k$ is a $C^\infty$ vector bundle over $\mathcal{U}_{d_u-k}$, which is written in the form 
\[
\hat{V}_k=\hat{V}\otimes |\Omega|
\]
with setting $\hat{V}_k=\pi^*_{d_u-k,0}(\hL)\otimes  \pi^*_{d_u-k,0}(V_G^*)^{\wedge k}$ and $|\Omega|=|\Omega_{d_u-k}|$.
\item $F^t_{k}:\hat{V}_k\to\hat{V}_k$ is a $C^\infty$ vector bundle map over the semi-flow $\hat{f}^t_{d_u-k}$ and induces the one-parameter semi-group of transfer operators $\hat{\cL}^t_{k}:\Gamma^\infty(\hat{V}_k)\to \Gamma^\infty(\hat{V}_k)$.
\end{enumerate}
Hence we can get the following theorem  as an application of  the argument in the previous subsections.
\begin{Theorem}\label{th:main_lift0} Suppose that $0\le k\le d_u$ and $s<0<r$.  The transfer operator $\hat{\cL}^t_k: \Gamma^{-\infty}(\hat{V}_k)\to \Gamma^{-\infty}(\hat{V}_k)$ restricts to a strongly continuous family of bounded operators $\hat{\cL}^t_k: \mathcal{H}^{s,r}(\hat{V}_k)\to \mathcal{H}^{s,r}(\hat{V}_k)$. 
Its generator ${A}_k$ has discrete spectrum on the region 
\[
R_{s,r}:=\{s\in \complex\mid \Re(s)> -\min\{|s|, r\}\cdot  \chi +C\}
\]
for some $C>0$ independent of $s$ and $r$. The exterior derivative operator $\hder_k^u:\Gamma^{-\infty}(\hat{V}_{k})\to \Gamma^{-\infty}(\hat{V}_{k+1})$ for $0\le k< d_u$ restricts to a bounded operator
\[
\hder_k^u:\mathcal{H}^{s,r}(\hat{V}_k)\to \mathcal{H}^{s',r'}(\hat{V}_{k+1})
\]
provided that $s'\le s-1$ and $0<r'\le r-1$. 
\end{Theorem}
\begin{proof} The former claims on $\hat{\cL}^t_k$ is a consequence of Proposition \ref{th:standard}. To see the latter claim on the exterior derivative operator $\hder_k^u$, we note that it is written as a composition of a first order differential operator with $C^\infty$ coefficient and the push-forward operator by a projection in \eqref{eq:Jseq}. Then Proposition \ref{pp:diffop} (and a simple fact on the latter  operator) gives the claim. 
\end{proof}

From Theorem \ref{pp:diff},  the operator $(\iota_u)_*:\Gamma^\infty_{cu}(V_k)\to \Gamma^{-\infty}(\hat{V}_k)$ in \eqref{eq:iotau} restricts to a continuous operator  
\[
(\iota_u)_*:\Gamma^\infty_{cu}(V_k)\to \mathcal{H}^{s,r}(\hat{V}_k).
\]
Similarly to the definition of $\mathcal{H}^{s,r}(W)$ in the last subsection, we
define $\mathcal{H}^{s,r}(V_k)$ as the completion of the space $\Gamma^\infty_{cu}(V_k)$ with respect to the pull-back of the norm on $\mathcal{H}^{s,r}(\hat{V}_{k})$ by $(\iota_u)_*$. 
If we choose the sequence of real parameters $s_k$ and $r_k$ for $0\le k\le d$, so that 
\[
s_k<-1<0<r_k,\qquad s_{k+1}\le s_k -1,\qquad r_{k+1}\le r_k-1,
\]
then, from Theorem \ref{pp:diff},  
the diagram \eqref{cd:extreq} extends to the following commutative diagram of bounded operators:
\begin{equation*}
\adjustbox{scale=0.95, center}{
\begin{tikzcd}
& \mathcal{H}^{s_k,r_k}(\hat{V}_{k}) \arrow[rr,"\hder_{k}^u"] \arrow[from=dd, "(\iota_{u})_*" near start] 
  & & \mathcal{H}^{s_{k+1},r_{k+1}}(\hat{V}_{k+1})  \arrow[from=dd,"(\iota_{u})_*" near start]
  \\ \mathcal{H}^{s_k,r_k}(\hat{V}_{k}) \arrow[rr, "\hder_k^u" near end, crossing over]\arrow[ur,"\hat{\cL}^t_{k}"] 
  & & \mathcal{H}^{s_{k+1},r_{k+1}}(\hat{V}_{k+1})\arrow[ur, "\hat{\cL}^t_{k}"]\\
& \mathcal{H}^{s_k,r_k}(V_{k}) \arrow[rr, "\der^u_{k}" near start] & & \mathcal{H}^{s_{k+1},r_{k+1}}(V_{k+1}) \\
\mathcal{H}^{s_k,r_k}(V_{k}) \arrow[rr, "\der^u_{k}"] \arrow[ur, "\cL^t_k"] \arrow[uu, "(\iota_{u})_*" near start ] & & \mathcal{H}^{s_{k+1},r_{k+1}}(V_{k+1})\arrow[ur, "\cL^t_{k-1}"]\arrow[uu,"(\iota_{u})_*" near start, crossing over]\\
\end{tikzcd}}
\end{equation*}
Therefore, setting $
\Lambda_k=\mathcal{H}^{s_k,r_k}(V_k)$ and connecting the diagram on the bottom face of the diagram above, we obtain the commutative diagram
\[
\begin{CD}
0@>>> \Lambda_0 @>{\der^u_0}>> \Lambda_1 @>{\der^u_1}>>\cdots @>{\der_{d_u-2}}>> \Lambda_{d_u-1} @>{\der_{d_u-1}}>> \Lambda_{d_u} @>>> 0 \\
@. @V{\cL^t_0}VV @V{\cL^t_1}VV  @. @V{\cL^t_{d_u-1}}VV @V{\cL^t_{d_u}}VV
\\
0@>>> \Lambda_0 @>{\der^u_0}>> \Lambda_1 @>{\der^u_1}>>\cdots @>{\der_{d_u-2}}>> \Lambda_{d_u-1} @>{\der_{d_u-1}}>> \Lambda_{d_u} @>>> 0
\end{CD}
\]
This finishes the proof of the claim (i) and (ii) of Theorem \ref{th:main}.  Finally the last claim on the dynamical Fredholm determinant is a direct consequence of Proposition \ref{pp:lf2}. 

\section{Fiber-wise cohomology}\label{sec:fiber}
In this section we prove Proposition \ref{pp:lf2}.
\subsection{Setting}
  Below we consider the setting that we introduced in Subsection \ref{ss:aniso}. We write $\Omega/|\Omega|$ for the orientation line bundle of $\ker D\pi_{B}$ defined in \eqref{eq:Tperp}.
For $0\le \nu \le d^\perp$, we set
\[
\Omega^{(\nu)}=((\ker D\pi_{B})^*)^{\wedge \nu}\quad\text{and}\quad \hat{W}_\nu=\hat{V}\otimes 
(\Omega/|\Omega|)\otimes \Omega^{(\nu)}
\]
so that $\hat{W}_{d^\perp}=\hat{W}$. 
The natural push-forward action of the flow $\hat{f}^t$ on $(\Omega/|\Omega|)\otimes \Omega^{(\nu)}$ is denoted 
\[
(\hat{f}^{-t})^*:(\Omega/|\Omega|)\otimes \Omega^{(\nu)}\to (\Omega/|\Omega|)\otimes \Omega^{(\nu)}.
\] 
Then the one-parameter family of the vector bundle maps 
\[
\hF^t_\nu=\hF^t\otimes (\hat{f}^{-t})^*: \hat{W}_\nu\to \hat{W}_\nu
\]
induces the transfer operator $
\hat{L}^t_{\nu}:\Gamma^{\infty}_{0}(\hat{W}_\nu)\to \Gamma^{\infty}_{0}(\hat{W}_\nu)$.
By Theorem \ref{th:standard}, this extends to a bounded operator 
\begin{equation}\label{eq:Ltnu}
\hat{L}^t_\nu:\mathcal{H}^{s,r}(\hat{W}_\nu)\to \mathcal{H}^{s,r}(\hat{W}_\nu).
\end{equation}
The dynamical Fredholm determinant  of $\hat{L}_\nu^t$ is defined formally by
\[
d_\nu(s)
= \exp\left(-\sum_{\po\in \PO} \sum_{m=1}^\infty 
\frac{e^{-sm|\po|}\cdot \mathrm{Tr}\, \hat{W}_{\po}\cdot \sigma^\perp(\po)\cdot \mathrm{Tr} ((P^\perp_\po)^{-m})^{\wedge \nu}}{m\cdot |\det(\mathrm{Id}-P_\po^{-m})|\cdot |\det(\mathrm{Id}-(P^\perp_\po)^{-m}|} \right)
\]
where  $P_\po^\perp$ is the transversal Jacobian matrix $P_{\hat{\po}}$ of the flow $\hat{f}^t$ along the periodic orbit $\hat{\po}=\iota(\po)$  restricted to the subspace $\ker D\pi_B$ and we set
\[
\sigma^\perp(\po)=\det P^\perp_\po/ |\det P^\perp_\po|.
\] 
Since  $\hat{L}_\nu^t$ is a smooth transfer operators associated to the hyperbolic flow $\hat{f}^t$, we can prove
\begin{Proposition}\label{pp:dynfdspec}
For each $0\le \nu \le d^\perp$, the dynamical Fredholm determinant $d_\nu(s)$ extends to an entire function and its
 zeros coincide with the discrete eigenvalues of the generator of $\hat{L}^t_{\nu}$, up to multiplicity. 
\end{Proposition}
For the proof of this proposition, we refer \cite{DyatlovGuillamouAxiomA}, though either of the methods in  \cite{DZ16, MR3648976, GLP12} will be applicable to check this claim. 

By algebraic computation, we see that $d(s)$ defined in \eqref{eq:dsabs} is expressed as the alternating product of  the dynamical Fredholm determinants $d_\nu(s)$ as 
\begin{equation}\label{eq:ds}
d(s)=\prod_{\nu=0}^{d^\perp}d_{\nu}(s)^{(-1)^{d^\perp-\nu}}.
\end{equation}
Hence, from Proposition \ref{pp:dynfdspec}, we see that  $d(s)$ has a meromorphic extension to the complex plane $\mathbb{C}$. Note however that Proposition \ref{pp:lf2} claims that $d(s)$ has no pole. To prove this, we have to see how the zeros of $d_{\nu}(s)$ cancel each other in the alternative product above, by developing a cohomological argument along the fibers of the fibration $\pi_B:B\to M$. 
 
\subsection{Nested invariant subspaces in $ \mathcal{H}^{s,r}(\hat{W}_{\nu})$}
We consider the exterior derivative operators along the fibers of the fibration $\pi_B:B\to M$.
To begin with, we fix an arbitrary $C^\infty$ flat connection $H$ of the vector bundle $\hat{V}$ along the fibers of the fibration $B$, which is not necessarily  invariant with respect to the action of $F^t_\nu$. 
On each of the fibers of the fibration $B$, we trivialize the vector bundle $\hat{V}$ so that the flat connection $H$ looks  trivial in it. Then, with respect to such trivializations,  we define the exterior derivative operators,
\[
\der^\perp_\nu:\Gamma^{\infty}(\hat{W}_\nu)\to \Gamma^{\infty}(\hat{W}_{\nu+1}), \quad 0\le \nu\le d^\perp-1
\]
along  the fibers of the fibration $B$. Similarly to Proposition \ref{pp:diffop}, they extend to bounded operators
\[
\der^\perp_\nu:\mathcal{H}^{s-\nu,r-\nu}(\hat{W}_\nu)\to \mathcal{H}^{s-\nu-1,r-\nu-1}(\hat{W}_{\nu+1}), \quad 0\le k\le d^\perp-1.
\]
Clearly we have $\der^\perp_{k+1}\circ \der^\perp_k=0$. But note that  the diagram of operators
\begin{equation}\label{cd:noncom}
\begin{CD}
\Gamma_{0}^{\infty}(\hat{W}_0) @>{\der^\perp_0}>> \Gamma_{0}^{\infty}(\hat{W}_1)@>{\der^\perp_1}>> \cdots @>{\der^\perp_{d^\perp-1}}>>  \Gamma_{0}^{\infty}(\hat{W}_{d^\perp})\\
@V{\hat{L}^t_0}VV @V{\hat{L}^t_1}VV @.@V{\hat{L}^t_{d^\perp}}VV\\
\Gamma_{0}^{\infty}(\hat{W}_0) @>{\der^\perp_1}>> \Gamma_{0}^{\infty}(\hat{W}_1)@>{\der^\perp_2}>> \cdots @>{\der^\perp_{d^\perp-1}}>>  \Gamma_{0}^{\infty}(\hat{W}_{d_B})
\end{CD}
\end{equation}
 is  \emph{not necessarily commutative} because the flat connection $H$ on $\hat{V}$ may not be invariant with respect to $\hat{F}^t_\nu$.

Below  we take a nested subspaces in $\Gamma_{0}^{\infty}(\hat{W}_\nu)$ which are, in a sense, skeletons of the space $\Gamma_{0}^{\infty}(\hat{W}_\nu)$ with respect to the spectral property of  $\hat{L}^t_\nu$. 
Let us consider a local trivialization of the fibration $B$ on $U\subset M$,
\begin{equation}\label{eq:kappa}
\kappa:\pi_B^{-1}(U)\to U\times \mathbb{D}
\end{equation}
where $\mathbb{D}\subset \real^{d^\perp}$ is the unit disk, 
and also that of the vector bundle $\hat{V}$
\begin{equation}\label{eq:hkappa}
\hat{\kappa}: \pi_{\hat{V}}^{-1}(\pi_{B}^{-1}(U))\to U\times \mathbb{D}\times \complex^{d_V} 
\end{equation}
where $d_V$ denotes the dimension of the vector bundle $\hat{V}$. We assume that the  local trivialization $\hat{\kappa}$ is taken so that the flat connection $H$ viewed in it  is trivial connection on the trivial vector bundle $(U\times \mathbb{D})\times \complex^{d_V}$.  

Suppose that the image of the invariant section $\iota:M\to B$ in the local chart $\kappa$ is the graph of the section
\begin{equation}\label{eq:iotai}
\iota_{\kappa}:U\to U\times \mathbb{D}.
\end{equation}
Let $\mu_{\kappa}$ be the measure on $U\times \mathbb{D}$ obtained as the image of a smooth Riemann volume on $U$ by $\iota$. We write $(y_1,y_2,\cdots, y_{d^\perp})$ for the standard coordinates on $\mathbb{D}\subset \real^{d^\perp}$ and write $\partial_{j}$ for the partial differential operator by the variable $y_j$ acting of the sections of the trivial vector bundle $(U\times \mathbb{D}^{d^\perp})\times \complex^{d_V}$. 
\begin{Remark}\label{rem:ntcon}
 If we consider a $C^\infty$ flat connection on the trivial vector bundle $(U\times \mathbb{D}^{d^\perp})\times \complex^{d_\nu}$ other than the trivial one, the  partial differential operator $\partial'_i$ corresponding to $\partial_{i}$ will be written
\[
\partial'_i u=\partial_i u+\psi u\quad \text{for }u\in \Gamma^\infty((U\times \mathbb{D})\times \complex^{d_V}),
\]
where $\psi:U\times \mathbb{D}^{d^\perp}\to L(\complex^{d^\perp})$ denotes the  $C^\infty$ function which takes values in the space $L(\complex^{d^\perp})$ of linear transformation on  $\complex^{d_V}$.
\end{Remark}

For a multi-index $\alpha=(\alpha(i))_{i=1}^{d^\perp}\in \mathbb{Z}_+^{d^{\perp}}$, we set 
$|\alpha|=\sum_{i=1}^{d^\perp} \alpha(i)$ and write $\partial^\perp_{\alpha}$ for the partial differential operator
\[
\partial^\perp_{\alpha}=\partial^{\alpha(1)}_{1}\partial^{\alpha(2)}_{2}\cdot \partial^{\alpha(d^\perp)}_{\nu}
\] 
acting on the section of the trivial vector bundle $(U\times \mathbb{D})\times \complex^{d_\nu}$.
For a subset $I=\{i(1)<i(2)<\cdots<i({q})\}$ of $\{1,2,\cdots, d^\perp\}$, we set $|I|=\nu$ and consider the differential form  
\begin{equation}\label{eq:dyi}
dy_I=dy_{i(1)}\wedge dy_{i(2)}\wedge \cdots \wedge dy_{i(\nu)}\in  ((\complex^{d^\perp})^*)^{\wedge \nu}.
\end{equation}
We will understood it as a constant section of $(U\times \mathbb{D})\times ((\complex^{d^\perp})^*)^{\wedge \nu}$. These give a local trivialization of the vector bundle $\Omega^{(\nu)}$. Hence, together with the local trivialization $\hat{\kappa}$ of the vector bundle $\hat{V}$, we obtain the local trivialization of the vector bundle\footnote{Here and henceforth we ignore the orientation bundle $\Omega/|\Omega|$  because it does not play any role in it} $\hat{W}_\nu$:
\[
\hat{\kappa}_\nu: \pi_{\hat{W}_\nu}^{-1}(\pi_{B}^{-1}(U))\to (U\times \mathbb{D})\times  (\complex^{d_V}\otimes ((\complex^{d^\perp})^*)^{\wedge \nu}).
\]

For integers $q\ge 0$ and $0\le \nu\le d^\perp$, we write $\mathbf{D}_{\nu,q}$ for the space of distributional section $u\in \Gamma^{-\infty}(\hat{W}_\nu)$ that is written in the form 
\begin{equation}\label{eq:localDq}
\sum_{|\alpha|\le q}\sum_{|I|=\nu}\sum_{j=1}^{d_V}
 (\varphi_{\alpha,j,I}\cdot \partial^\perp_{\alpha}  \mu)\cdot ( \mathbf{e}_j\otimes dy_I)
\end{equation}
in each of the local trivializations as above, 
where $\{\mathbf{e}_j\}_{j=1}^{d_V}$ is the standard basis of $\complex^{d_V}$ and $\varphi_{\alpha,j,I}:U\times \mathbb{D}\to \complex^{d_V}$ are vector-valued function in $\Gamma^\infty_{cu}(M)^{d_V}$ with compact support. (Recall Remark~\ref{rm:gammauM} for the definition of $\Gamma^\infty_{cu}(M)$.)  From the fact noted in Remark \ref{rem:ntcon}, the spaces $\mathbf{D}_{\nu,q}$ actually does not depend on the choice of the flat connection $H$ on $\hat{W}_\nu$ and are invariant with respect to the action of  $\hat{L}^t_\nu:\Gamma^{-\infty}(\hat{W}_\nu)\to \Gamma^{-\infty}(\hat{W}_\nu)$.  Further, in the same manner as in the proof of Proposition \ref{pp:diff}, we can prove the next lemma. 
\begin{Lemma}\label{lm:sr} If  $s<0$ is sufficiently small according to $q$, the space $\mathbf{D}_{\nu,q}$ is contained in $\mathcal{H}^{s,r}(\hat{W}_\nu)$. 
\end{Lemma}

Below we fix some integer $q_0>0$ and consider the nest of subspaces  
\begin{equation}\label{eq:nestsp1}
\mathbf{D}_{\nu,0}\subset \mathbf{D}_{\nu,1}\subset \cdots \subset \mathbf{D}_{\nu,q_0}\qquad\text{in $\Gamma^{-\infty}(\hat{W}_\nu)$. }
\end{equation}
Thanks to Lemma \ref{lm:sr}, we can take and fix $s<0<r$ so that we have $\mathbf{D}_{\nu,q_0}\subset \mathcal{H}^{s-\nu,r-\nu}(\hat{W}_\nu)$ for $0\le \nu\le d^{\perp}$. For $0\le q\le q_0$, we write $\overline{\mathbf{D}_{\nu,q}}$ for the closure of $\mathbf{D}_{\nu,q}$ in $\mathcal{H}^{s-\nu,r-\nu}(\hat{W}_\nu)$, so that we obtain the nest of closed subspaces
\begin{equation}\label{eqnestsp2}
\overline{\mathbf{D}_{\nu,0}}\subset \overline{\mathbf{D}_{\nu,1}}\subset \cdots \subset \overline{\mathbf{D}_{\nu,q_0}}\qquad\text{in $\mathcal{H}^{s-\nu,r-\nu}(\hat{W}_\nu)$. }
\end{equation}
The exterior derivative operator $\der^\perp_\nu$ maps $\mathbf{D}_{\nu,q}$ into $\mathbf{D}_{\nu+1,q+1}$ and, by continuity, maps $\overline{\mathbf{D}_{\nu,q}}$ into $\overline{\mathbf{D}_{\nu+1,q+1}}$.
From Remark \ref{rem:ntcon}, the transfer operator $\hat{L}^t_\nu$ preserves the spaces $\mathbf{D}_{\nu,q}$ and $\overline{\mathbf{D}_{\nu,q}}$. Hence the action of $\hat{L}^t_\nu$ on the space $\overline{\mathbf{D}_{\nu,q_0}}$ is ``triangular" in the sense that it preserves the nest of subspaces \eqref{eqnestsp2} (and \eqref{eq:nestsp1}). 
Further, from Remark \ref{rem:ntcon}, we see that the image of 
\[
\hat{L}^t_{\nu+1}\circ \der^\perp_\nu- \der^\perp_{\nu} \circ \hat{L}^t_{\nu}:\overline{\mathbf{D}_{\nu,q}}\to \overline{\mathbf{D}_{\nu+1,q+1}}
\]
is contained in $\overline{\mathbf{D}_{\nu+1,q}}$ and therefore the following diagram commutes:
\begin{equation}\label{eq:comm}
\begin{CD}
\overline{\mathbf{D}_{\nu,q}}/\overline{\mathbf{D}_{\nu,q-1}} @>{\der^\perp_\nu}>> \overline{\mathbf{D}_{\nu+1,q+1}}/\overline{	\mathbf{D}_{\nu+1,q}}\\
@V{\hat{L}^t_{\nu}}VV @V{\hat{L}^t_{\nu+1}}VV\\
\overline{\mathbf{D}_{\nu,q}}/\overline{\mathbf{D}_{\nu,q-1}} @>{\der^\perp_\nu}>> \overline{\mathbf{D}_{\nu+1,q+1}}/\overline{\mathbf{D}_{\nu+1,q}}
\end{CD}
\end{equation}
even though the diagram \eqref{cd:noncom} may not commutes. 
Recalling the argument in the proof of Proposition \ref{pp:diff} and the definition of $\mathcal{H}^{s,r}(V)$, we see that
\begin{equation}
\overline{\mathbf{D}_{d^\perp,0}}=\overline{\iota_*(\Gamma^{\infty}_{cu}(V))}=\iota_*(\mathcal{H}^{s-d^\perp,r-d^\perp}(V)).
\end{equation}

\subsection{Approximation by the nested subspaces $\mathbf{D}_{\nu,q}$}
In this subsection, we show that that the nested subspaces $\mathbf{D}_{\nu,q}$ exhaust $\cH^{s-\nu,r-\nu}(\hat{W}_\nu)$  in relation to the spectral properties of the operators $\hat{L}^t_\nu$. 
Let $K>0$ is an arbitrary positive real number, we prove that the next proposition holds true if we  choose $q_0>0$ and $s<0<r$ according to $K$.   
\begin{Proposition}\label{prop:approx} Let $0\le \nu\le d^{\perp}$ and let  $\sigma_0\in \complex$ with $\Re(\sigma_0)\ge -K$ be one of the discrete eigenvalues of the generator of the one-parameter semi-group $
\hat{L}^t_{\nu}:\mathcal{H}^{s-\nu,r-\nu}(\hat{W}_\nu)\to \mathcal{H}^{s-\nu,r-\nu}(\hat{W}_\nu)$. Then  the  generalized eigenspace $E$ for the eigenvalue $\sigma_0$ is contained in $\overline{\mathbf{D}_{\nu,q_0}}$. 
\end{Proposition}
\begin{proof}
Let $T>0$ be a positive real number and let $\psi:\real\to \real$ be a $C^\infty$ non-negative-valued function supported in $[T,T+1]$ satisfying $\int \psi dx=1$. 
Instead of the generator, we consider the bounded operator
\begin{equation}\label{eq:M}
R=\int \psi(t)e^{-i\cdot \mathrm{Im}(\sigma_0)t}\cdot  \hat{L}^t_{\nu}\,dt:\mathcal{H}^{s-\nu,r-\nu}(\hat{W}_\nu)\to \mathcal{H}^{s-\nu,r-\nu}(\hat{W}_\nu).
\end{equation}
Then $E$  is contained in the generalized eigenspace of $R$ for the eigenvalue 
\[
\mu_0=\int \psi(t)e^{-i\cdot \mathrm{Im}(\sigma_0)t}\cdot  e^{\sigma_0 t} dt=\int \psi(t)e^{\mathrm{Re}(\sigma_0)t} dt>0.
\]
Letting $s<0$ be sufficiently small and $r>0$ be sufficiently large, we may and do suppose that the essential spectral radius of the operator $M$ is smaller than~$\mu_0$. 
We split the spectral set of $M$ into the finite set of discrete eigenvalues of $M$  whose absolute values are not smaller than $\mu_0$ and the rest. We write  $\mathcal{H}^{s-\nu,r-\nu}(\hat{V}_\nu)=H_{\ge }\oplus H_{<}$ for the corresponding spectral decomposition, so that $H_{\ge}$ is of finite dimensional and contains $u$. 
Since the space $\Gamma^\infty_{0}(\hat{V}_\nu)$ of $C^\infty$ sections with compact support is dense in $\mathcal{H}^{s-\nu,r-\nu}(\hat{W}_\nu)$ by definition, the spectral projector $P_{\ge 0}$ to $H_{\ge}$ (along $H_{<}$) is a surjection from $\Gamma^\infty_{0}(\hat{V}_\nu)$ to $H_{\ge}$. Hence, for any $u\in E$, we can take and fix a smooth section $\tilde{u}\in \Gamma^\infty_{0}(\hat{V}_\nu)$ such that $P_{\ge 0}\tilde{u}=u$. This implies that  
\begin{equation}\label{eq:approx1}
\lim_{m\to \infty} \mu_0^{-m} M^{m}\tilde{u}=u \quad\text{in }\mathcal{H}^{s-\nu,r-\nu}(\hat{W}_\nu).
\end{equation}
For the proof of the proposition, it is enough to show the claim that, for any $\varepsilon>0$, there exists arbitrarily large $m>0$ and $v\in \mathbf{D}_{\nu,q}$ such that 
\begin{equation}\label{eq:claim_approx}
\|\mu_0^{-m}M^m\tilde{u}-v\|_{\mathcal{H}^{s-\nu,r-\nu}(\hat{W}_\nu)}< \varepsilon.
\end{equation}
Below we construct such an element $v\in \mathbf{D}_{\nu,q}$ assuming that $m$ is sufficiently large.
The point in the proof is that the support of  the section $M^m\tilde{u}$ is contained in the small neighborhood
\[
\mathcal{N}_{m,i}=\{b\in B\mid 
d(b, \iota(\pi_B(b)))<Ce^{-\chi mT}\}
\]
of the attractor $\mathcal{A}=\iota(M)$, because the semi-flow $\hat{f}^t$ is exponentially contracting along the fibers of the fibration $\pi_B:B\to M$. 

Let us recall the local trivializations $\kappa$ and $\hat{\kappa}$ in \eqref{eq:kappa} and \eqref{eq:hkappa}.
Then the section $M^m\tilde{u}$ viewed in the local trivialization $\hat{\kappa}$ is expressed as a $C^\infty$ section 
\[
h:U\times \mathbb{D}\to \complex^{d_V}\otimes ((\complex^{d^\perp})^*)^{\wedge \nu}).
\]
For each multi-index $\alpha=(\alpha_j)_{j=1}^{d^\perp}\in \mathbb{Z}_+^{d^\perp}$, we define the section 
\[
h_\alpha:U\times \mathbb{D}\to \complex^{d_V}\otimes ((\complex^{d^\perp})^*)^{\wedge \nu})
\]
by 
\[
h_{\alpha}(x,y)=\frac{1}{\alpha!}\int (y'-y)^{\alpha} h(x,y') dy':U\times \mathbb{D}\to \complex^{d_V}\otimes ((\complex^{d^\perp})^*)^{\wedge \nu}) 
\]
where $(x,y)$ denotes the coordinates on $U\times \mathbb{D}$ and $\alpha !=\alpha_1!\cdot \alpha_2!\cdots \alpha_{d^{\perp}}!$. Then we define a distributional section 
\[
v_{\kappa}:U\times \mathbb{D}\to \complex^{d_V}\otimes ((\complex^{d^\perp})^*)^{\wedge \nu})
\]
by
\begin{align*}
v_{\kappa}(x,y)&=\sum_{|\alpha|\le q_0}    \partial^{\alpha}_y (h_{\alpha}(x,y)\cdot \delta(y-\iota(x)))\\
&=\frac{1}{\alpha !}\sum_{|\alpha|\le q_0} \partial^{\alpha}_y \int  (y'-y)^\alpha \cdot h(x,y')  \cdot  \delta(y-\iota(x)) dy'.
\end{align*}
We define a distributional section $v\in \Gamma^{-\infty}(\hat{W}_\nu)$ by patching the distributional section $v_{\kappa}$ for local trivializations $\hat{\kappa}$ that cover $B$, using a $C^\infty$ partition of unity subordinate to it. Clearly $v$ belongs to $\mathbf{D}_{\nu,q_0}$. 

We check the condition \eqref{eq:claim_approx} for the section $v\in \mathbf{D}_{\nu,q_0}$ defined above. Note that distributional sections of the vector bundle $\hat{W}_k$ is an element of the dual of the space $\Gamma^\infty_{0}(\hat{W}_\nu^\dag)$ with setting 
\[
\hat{W}_\nu^\dag=\hat{V}^*\otimes \Omega^{(d^\perp-\nu)}\otimes \pi_B^*(|\Det^*|).
\]

 Below we confuse the (distributional) sections $v$ and its expression in the local trivialization $\hat{\kappa}_\nu$ and similarly for the section $w\in \Gamma^\infty_{0}(\hat{W}_k^\dag)$. Then the pairing $\langle v,w\rangle$ is computed in the local trivialization as 
 \begin{align*} 
&
\int dx dy dy' \cdot \sum_{|\alpha|\le q_0}\frac{(y'-y)^\alpha}{\alpha !} \cdot \partial^\alpha_y w(x,y) \cdot  h(x,y')  \cdot \delta(y-\iota(x)) \\
&
=\int dx dy'  \cdot \sum_{|\alpha|\le q_0}\frac{(y'-\iota(x))^\alpha}{\alpha !} \cdot \partial^\alpha_y w(x,\iota(x)) \cdot  h(x,y').
\end{align*}
By Taylor's theorem and the fact that $v$ (or $h$) is supported on $\mathcal{N}_{m,i}$, we can bound the quantity $|\left\langle M^m\tilde{u}-v, w\right\rangle |$ by 
\begin{align*}
&\left|\int   \bigg(w(x,y)-\sum_{|\alpha|\le q_0}\frac{(y-\iota(x))^\alpha}{\alpha !}\cdot \partial^\alpha_y w(x,\iota(x))\bigg)  h(x,y') \right|dx dy' 
\\
&\qquad\le C e^{-(q_0+1-B)\chi m T}\cdot \left(\max_{|\alpha|\le {q_0+1}}\|\partial^\alpha_y w\|_{L^\infty}\right)
\end{align*}
where we let $B>0$ be a constant depending on the flow $\hat{f}^t$ such that $\|M^m\tilde{u}(x,y')\|_{L^2}\le Ce^{B\chi mT}$.  

Now we suppose that we take $q_0$ sufficiently large so that $q_0>B$. Then 
the last estimate implies that the difference $M^m\tilde{u}-v$ is small when we look at its ``low frequency part". Indeed, if we take a small constant $\theta>0$ and set
\[
\chi_T:T^*M\to [0,1],\quad \chi_T(x,\xi)=\chi(\|e^{-\theta \chi T}\cdot\xi\|)
\] 
where $\chi:\real\to [0,1]$ is a smooth function which takes value $1$ on the region $|s|\le 1$ while $0$ on the region $|s|\ge 2$, we have that 
\[
\|\mathrm{Op}(\chi_T) (M^m\tilde{u}-v)\|^2_{L^2} \le  C e^{-((1-\theta)(q_0+1)-B)\chi m T}.
\]
Therefore we have the estimate
\[ 
\|\mathrm{Op}(W^{s,r}\cdot \chi_T) (M^m\tilde{u}-v)\|_{L^2} \le C e^{-((1-\theta)(q_0+1)-K-r\theta)\chi T}, 
\]
where the right-hand side is small provided that $\theta>0$ is sufficiently small. 

For the remaining ``high-frequency parts", we estimate each of $M^m\tilde{u}$ and $v$ as follows. For the former,  the estimate \eqref{eq:wsrmotive2} leads to 
\[
\|\mathrm{Op}(W^{s,r}\cdot (1-\chi_T)) M^m\tilde{u}\|^2_{L^2}\le  
 C e^{-(\min\{|s|, r\}\theta-B) \chi T}.
\]
For the latter, we recall the argument in the proof of Proposition \ref{pp:diff} and see that 
\[
\|\mathrm{Op}(W^{s,r}\cdot (1-\chi_T)) v\|^2_{L^2} \le C e^{(-|s|\theta+B') \chi m T}
\]
where $B'>0$ is another constant depending on the flow $\hat{f}^t$. 
Therefore if we let $|s|$ and $r$ be sufficiently large, we obtain 
\[
\|\mathrm{Op}(W^{s,r})( M^m\tilde{u}- v)\|^2_{L^2} \le C e^{(-\min\{|s|, r\}\theta/2) \chi m T}.
\]
Hence we obtain the required estimate if we let $T$ be sufficiently large. 
\end{proof}

\subsection{A version of Poincar\'e lemma}
Recall the \emph{non-commutative} diagram \eqref{cd:noncom}. It induces the following diagram of bounded operators, which is again not commutative:
\begin{equation}\label{cd:noncom2}
\begin{CD}
\mathcal{H}^{s,r}(\hat{V}_0) @>{\der^\perp_0}>> \mathcal{H}^{s-1,r-1}(\hat{V}_1)@>{\der^\perp_1}>> \cdots @>{\der^\perp_{d_B-1}}>>  \mathcal{H}^{s-d^\perp,r-d^\perp}(\hat{V}_{d_B})\\
@V{\hat{L}^t_0}VV @V{\hat{L}^t_1}VV @.@V{\hat{L}^t_{d_B}}VV\\
\mathcal{H}^{s,r}(\hat{V}_0) @>{\der^\perp_0}>> \mathcal{H}^{s-1,r-1}(\hat{V}_1)@>{\der^\perp_1}>> \cdots @>{\der^\perp_{d_B-1}}>>  \mathcal{H}^{s-d^\perp,r-d^\perp}(\hat{V}_{d_B})
\end{CD}
\end{equation}
Note that, in relation to the discrete eigenvalues in the region $\mathrm{Re}(s)>-K$, it is enough to consider the following restriction of the diagram above:
\begin{equation}\label{cd:noncom3}
\begin{CD}
\overline{\mathbf{D}_{0,q_0}}@>{\der^\perp_0}>> \overline{\mathbf{D}_{1,q_0+1}}@>{\der^\perp_1}>> \cdots @>{\der^\perp_{d_B-1}}>>  \overline{\mathbf{D}_{d^\perp,q_0+d^\perp}}\\
@V{\hat{L}^t_0}VV @V{\hat{L}^t_1}VV @.@V{\hat{L}^t_{d_B}}VV\\
\overline{\mathbf{D}_{0,q_0}}@>{\der^\perp_0}>> \overline{\mathbf{D}_{1,q_0+1}}@>{\der^\perp_1}>> \cdots @>{\der^\perp_{d_B-1}}>>  \overline{\mathbf{D}_{d^\perp,q_0+d^\perp}}
\end{CD}
\end{equation}
Note that, though this diagram is still not commutative, we obtain the following commutative diagram by connecting \eqref{eq:comm}:
\begin{equation}\label{eq:comm01}
\begin{CD}
\overline{\mathbf{D}_{\nu-1,q-1}}/\overline{\mathbf{D}_{\nu,q-2}} @>{\der^\perp_{\nu-1}}>>
\overline{\mathbf{D}_{\nu,q}}/\overline{\mathbf{D}_{\nu,q-1}} @>{\der^\perp_\nu}>> \overline{\mathbf{D}_{\nu+1,q+1}}/\overline{	\mathbf{D}_{\nu+1,q}}\\
@V{\hat{L}^t_{\nu-1}}VV @V{\hat{L}^t_{\nu}}VV @V{\hat{L}^t_{\nu+1}}VV\\
\overline{\mathbf{D}_{\nu-1,q-1}}/\overline{\mathbf{D}_{\nu,q-2}} @>{\der^\perp_\nu}>>
\overline{\mathbf{D}_{\nu,q}}/\overline{\mathbf{D}_{\nu,q-1}} @>{\der^\perp_{\nu-1}}>> \overline{\mathbf{D}_{\nu+1,q+1}}/\overline{	\mathbf{D}_{\nu+1,q}}\end{CD}
\end{equation}
Since $\der^\perp_\nu\circ \der^\perp_{\nu-1}=0$, the transfer operator $\hat{L}^t_{\nu}$ acts on the cohomological space
\[
\hat{L}^t_{\nu}:\ker \der^{\perp}_\nu/\overline{\mathrm{Im} \der^\perp_{\nu-1}}\to \ker \der^{\perp}_\nu/\overline{\mathrm{Im} \der^\perp_{\nu-1}}
\]
where we understand that $\der^\perp_{\nu-1}$ and $\der^\perp_{\nu}$ are those in the diagram \eqref{eq:comm01} and the closures are taken in $\overline{\mathbf{D}_{\nu,q}}/\overline{\mathbf{D}_{\nu,q-1}}$. 

In the next (or the last) subsection,  we will introduce a formal adjoint of the operator $\der^\perp_\nu$,
\begin{equation}\label{eq:cod0}
\der^*_\nu:\mathbf{D}_{\nu,q}\to \mathbf{D}_{\nu-1,q-1},
\end{equation}
and prove the following proposition. In the statement, we write $\overline{\overline{\mathbf{D}_{\nu,q}}}$ for the closure of $\mathbf{D}_{\nu,q}$ in the space $\mathcal{H}^{s-\nu-1, r-\nu-1}(\hat{W}_\nu)$ and equipped it with the topology inherited from $\mathcal{H}^{s-\nu-1, r-\nu-1}(\hat{W}_\nu)$. Note that we have $
\overline{\mathbf{D}_{\nu,q}}\subset \overline{\overline{\mathbf{D}_{\nu,q}}}$.
\begin{Proposition}\label{pp:deltastar}
 The operator $\der^*_\nu$ extends to a bounded operator
\[
\der^*_\nu:\overline{\mathbf{D}_{\nu,q}}\to \overline{\overline{\mathbf{D}_{\nu-1,q-1}}}\subset \mathcal{H}^{s-\nu, r-\nu}(\hat{V}_\nu).
\]
 The operator
\begin{equation}\label{eq:deltastar0}
\der^*_{\nu+1}\circ \der^\perp_{\nu}-\der^\perp_{\nu-1}\circ \der^*_{\nu}-(q+d^\perp-\nu)\cdot \mathrm{Id}:\mathbf{D}_{\nu,q}\to \mathbf{D}_{\nu,q}
\end{equation}
maps $\mathbf{D}_{\nu,q}$ into $\mathbf{D}_{\nu,q-1}\subset \mathbf{D}_{\nu,q}$ and extends to a bounded operator
\begin{equation}\label{eq:deltastar}
\der^*_{\nu+1}\circ \der^\perp_{\nu}-\der^\perp_{\nu-1}\circ \der^*_{\nu}-(q+d^\perp-\nu)\cdot \mathrm{Id}:\overline{\mathbf{D}_{\nu,q}}\to \overline{\overline{\mathbf{D}_{\nu,q}}}.
\end{equation}
\end{Proposition}
Below we prove Proposition \ref{pp:lf2} assuming the proposition above. Recall Proposition \ref{pp:dynfdspec} and the expression \eqref{eq:ds} of $d(s)$. 
Recall also that the zeros of the dynamical Fredholm determinant $d_\nu(s)$ in the half-plane $\mathrm{Re}(s)>-K$ coincide with the discrete eigenvalues of the generator $A_\nu$ of $\hat{L}^t_{\nu}$ in \eqref{eq:Ltnu} and the multiplicity of each of such zeros coincides with the algebraic multiplicity of the corresponding eigenvalue. For a complex number $\sigma\in \complex$ with $\mathrm{Re}(s)>-K$, let $\Sigma_\nu(\sigma)$ be the generalized eigenspace of the generator $A_\nu$ for the eigenvalue $\sigma$ if $\sigma$ is an eigenvalue and $\Sigma_\nu(\sigma)=0$ otherwise. 
For the proof of Proposition \ref{pp:lf2}, it is enough to show that, for such a $\sigma\in \complex$, 
\begin{enumerate}
\item The following sequence is exact if $\nu<d^\perp$:
\[
\begin{CD}
\Sigma_{\nu-1}(\sigma)@>{\der^\perp_{\nu-1}}>> \Sigma_{\nu}(\sigma)@>{\der^\perp_{\nu}}>> \Sigma_{\nu+1}(\sigma)
\end{CD}
\]
\item For the case $\nu=d^\perp$, we have 
\[
\Sigma_{d^\perp}(\sigma)=\left(\Sigma_{d^\perp}(\sigma)\cap \overline{\mathbf{D}_{d^\perp,0}}\right) \oplus \der^\perp_{d^\perp-1}(\Sigma_{d^\perp-1}(\sigma)).
\]
\end{enumerate}
For $0\le q\le q_0$ and $0\le \nu\le d^\perp$, 
let  $\Sigma_{\nu,q}(\sigma)$ be the generalized eigenspace for an eigenvalue $\sigma$ of the generator $A_{\nu,q}$ of the transfer operators 
\[
L^t_{\nu}:\overline{\mathbf{D}_{\nu,q}}/\overline{\mathbf{D}_{\nu,q-1}}\to \overline{\mathbf{D}_{\nu,q}}/\overline{\mathbf{D}_{\nu,q-1}}.
\]
We show that 
\begin{equation}\label{eq:exactseq}
\begin{CD}
\Sigma_{\nu-1,q-1}(\sigma)@>{\der^\perp_{\nu-1,q-1}}>> \Sigma_{\nu,q}(\sigma)@>{\der^\perp_{\nu,q}}>> \Sigma_{\nu+1,q+1}(\sigma)
\end{CD}
\end{equation}
is exact except for the case $(\nu,q)=(d^\perp,0)$, where 
\[
\der^\perp_{\nu,q}:\overline{\mathbf{D}_{\nu,q}}/\overline{\mathbf{D}_{\nu,q-1}}\to \overline{\mathbf{D}_{\nu+1,q+1}}/\overline{\mathbf{D}_{\nu+1,q}}
\]
denotes the operator that $\der^\perp_\nu$ induces. Then, since the transfer operator $\hat{L}^t_\nu$ preserves the nest of closed subspaces \eqref{eqnestsp2}, the claim (i) follows. It also follows $\Sigma_{d^\perp}(\sigma)=\Sigma_{d^\perp,0}(\sigma)$ and therefore so does the claim (ii).


We prove exactness of \eqref{eq:exactseq}. Note that we may suppose that the finite dimensional subspace $\Sigma_{\nu,q}(\sigma)$ is contained in the closure of $\mathbf{D}_{\nu,q}$ in the space $\mathcal{H}^{s-\nu+1,r-\nu+1}(\hat{V}_\nu)$ because the generalized eigenspace $\Sigma_{\nu}(\sigma)$ is intrinsic to the transfer operator $\hat{L}^t_\nu$. 
Suppose that $v\in \Sigma_{\nu,q}(\sigma)$ is represented by an element $\tilde{v}\in \Sigma_{\nu}(\sigma)\cap \overline{\mathbf{D}_{\nu,q}}$. We assume that $\der^\perp_{\nu,q} v=0$ or, in other words, that  $\der^\perp_{\nu} \tilde{v}\in \overline{\mathbf{D}_{\nu+1,q}}$. 
If we apply the operator \eqref{eq:deltastar} in Proposition \ref{pp:deltastar} to $\tilde{v}$, we see
\[
\der^{\perp}_{\nu-1}\circ \der^*_{\nu}\tilde{v}-((q+d^\perp-\nu)) \tilde{v}\in \overline{\mathbf{D}_{\nu,q-1}}.
\] 
Then, from commutativity of  the diagram \eqref{eq:comm}, the element
\[
w=[\der^*_{\nu}\tilde{v}]\in \overline{\mathbf{D}_{\nu-1,q-1}}/\overline{\mathbf{D}_{\nu-1,q-2}}
\]
satisfies  $\der^\perp_{\nu-1,q-1}(w)=v$ provided $q+d^\perp-\nu>0$. 
This finishes the proof.

\subsection{Proof of Proposition \ref{pp:deltastar}}
We define the operator $\der^*_\nu$ in \eqref{eq:cod0} and prove Proposition \ref{pp:deltastar}.
We first consider the exterior derivative operator on  $\real^{d^\perp}$ and an adjoint of it. Let us write $\delta$ for the Dirac delta function at the origin $0$ on $\real^{d^\perp}$. For any multi-index $\alpha=(\alpha_i)_{i=1}^{d^\perp}$ with $\alpha_i\in \mathbb{Z}_+$, we set
\[
\delta^{(\alpha)}=\left(\frac{\partial}{\partial x_1}\right)^{\alpha_1}
\left(\frac{\partial}{\partial x_2}\right)^{\alpha_2}\cdots \left(\frac{\partial}{\partial x_{d^\perp}}\right)^{\alpha_{d^\perp}} \, \delta.
\]
We recall the notation $dy_{I}$ in \eqref{eq:dyi} and regard it as a constant differential form on $\real^{d^\perp}$.  
For $0\le k\le d^\perp$ and $q\ge 0$, let $\Omega_{\nu,q}$ be the linear space spanned by the differential $k$-forms (with distributional coefficients) 
\begin{equation}\label{eq:omega}
\delta^{(\alpha)} dy_{I}
\end{equation}
for $I=\{i_1<i_2<\cdots<i_\nu\}$ with $|I|=\nu$ and $\alpha\in \mathbb{Z}_+^{d^\perp}$ with $|\alpha|=q$. We define the linear operators
\begin{equation}\label{def:derd_fiber}
\derd_{\nu}:\Omega_{\nu,q}\to \Omega_{\nu+1,q+1},\qquad \derd^*_\nu:\Omega_{\nu,q}\to \Omega_{\nu-1,q-1}
\end{equation}
respectively as the linear extension of the following operation: 
\begin{align*}
&\derd_{\nu}(\delta^{(\alpha)} dy_{I})= \sum_{i=1}^{d^\perp}\delta^{(\alpha+i)} \cdot dy_i \wedge dy_{I}= \sum_{i=1}^{d^\perp}(-1)^{\sharp(I,i)} \delta^{(\alpha+i)} \cdot dy_{I\cup \{i\}}
\intertext{
and}
&\derd^*_{\nu}(\delta^{(\alpha)} dy_{I})= \sum_{i=1}^{d^\perp}  (-1)^{\sharp(I,i)-1} \alpha_i \cdot  \delta^{(\alpha-i)} \cdot  dy_{I\setminus\{i\}}
\end{align*}
where we use shorthand notation 
\[
\alpha\pm i=(\alpha_1, \cdots, \alpha_i\pm 1, \cdots, \alpha_{d^\perp})
\]
and we understand that $\delta^{(\alpha-i)}=0$ (resp. $dy_{I\setminus\{i\}}=0$) if $\alpha_i=0$ (resp. $i\notin I$).

\begin{Remark}\label{rem:delta} The operator $\derd_\nu$ is just a natural extension of the exterior derivative operator to currents. The operator $\derd^*_{\nu}$ can be understood as a kind of adjoint of the operator $\derd_\nu$. Indeed we can express it as 
\[
\derd^*_\nu=(-1)^{|I|} \cdot *^{-1}\circ \derd_{d^\perp-\nu} \circ *
\]
for the isomorphism
\[
*:\Omega_{m,q}\to \Omega_{d^\perp-m,q}^\dag,\quad *(\delta^{(\alpha)} dy_{I})= (-1)^{\sharp I}y^\alpha dy_{I^c}
\]
where $I^c=\{1,2,\cdots, d^\perp\}\setminus I$, $y^{\alpha}=y^{\alpha_1}_1y^{\alpha_2}_2\cdots y^{\alpha_{d^{\perp}}}_{d^\perp}$ and $\Omega_{m,q}^\dag$ denotes the linear space spanned by the differential form on $\real^{d^\perp}$ of the form $y^{\alpha} dy_I$ with $|\alpha|=q$ and $|I|=\nu$. ($\sharp I$ is taken so that $dy_I\wedge dy_{I^c}=(-1)^{\sharp I} dy_{1}\wedge \cdots \wedge dy_{d^\perp}$.) 
\end{Remark} 
We obtain the next lemma by computation. 
\begin{Lemma}\label{lem:exact} We have
\[
\derd_{\nu+1}\circ \derd_\nu=0,\qquad \derd^*_{\nu-1}\circ \derd^*_\nu=0
\]
and also 
\[
(\derd_{\nu+1}^*\circ \derd_\nu+\derd_{\nu-1}\circ \derd_{\nu}^*)|_{\Omega_{\nu,q}}=-(q+{d^\perp}-\nu)\cdot \mathrm{Id}|_{\Omega_{\nu,q}}.
\]
\end{Lemma}

Next we define the operators
\begin{equation}\label{def:derd_glob}
\derd_{\nu}:\mathbf{D}_{\nu,q}\to \mathbf{D}_{\nu+1,q+1},\qquad \derd^*_\nu:\mathbf{D}_{\nu,q}\to \mathbf{D}_{\nu-1,q-1}
\end{equation}
respectively as the application of the operators $\derd_{\nu}$ and $\derd_{\nu}^*$ in \eqref{def:derd_fiber} to each fiber $\pi_B^{-1}(x)$ with setting $\iota(x)$ as the origin. More precisely, we set up a finite system of local trivializations $\hat{\kappa}_j$, $1\le j\le J$, and introduce a $C^\infty$ partition of unity $\{\rho_j\}_{j\in J}$ on $B$ subordinate to the covering by the domain of the local trivializations $\hat{\kappa}_j$. We may and do assume that the functions $\rho_j$ are constant along the fibers of $B$. Then we set 
\[
\derd_{\nu}u=\sum_{j\in J} \derd_{\nu,j}(\rho_j u),\quad 
\derd_{\nu}^*u=\sum_{j\in J} \derd_{\nu,j}^*(\rho_j u)
\]
where $\derd_{\nu,j}$ and $\derd_{\nu,j}$ are the application of the operators in \eqref{def:derd_fiber} on each fiber  with setting $\iota(x)$ as the origin. 

\begin{Remark}\label{rem:derd_g_l}
The definition of the operator $\derd_{\nu}$ does not depend on the choice of local trivializations and make a  geometric sense, one because $\derd_{\nu}$ is the extension of the exterior derivative when we trivialize $\hat{V}$ using the flat connection $H$. But the operator $\derd_{\nu}^*$ will depend on the choice of local trivializations. 
\end{Remark} 

Since the claims of Lemma \ref{lem:exact} holds also for \eqref{def:derd_glob}, 
 Proposition \ref{pp:deltastar} follows if we prove the following lemma. 
\begin{Lemma} The latter operator $\derd_{\nu}^*$ in \eqref{def:derd_fiber} extends to a bounded operator 
\[
\derd_{\nu}^*:\overline{\mathbf{D}_{\nu,q}}\to \overline{\overline{\mathbf{D}_{\nu+1,q+1}}}.
\]
\end{Lemma}
\begin{proof}
From the definition of the operator $\derd_{\nu}^*$ using local trivialization, it is enough to show that the operator 
\[
D_i:\varphi(x,y) \delta^{\alpha}(y-\iota(x))\mapsto \varphi(x,y) \delta^{\alpha-i}(y-\iota(x)),\quad 1\le i\le d^\perp
\]
extends to a bounded operator from $\mathcal{H}^{s-\nu,r-\nu}(U\times \mathbb{D})$ to itself. But this is not difficult to check from the definition of the anisotropic Sobolev space $\mathcal{H}^{s-\nu,r-\nu}(U\times \mathbb{D})$ provided that $s$ is sufficiently small, because  $D_i$ is the extension of the operator
\[
{D}_i^{\circ}:\varphi(x,y_1,\cdots,y_i, \cdots, y_{d^\perp})\mapsto 
\int_{-\infty}^{y_i}\varphi(x,y_1,\cdots,y'_i, \cdots, y_{d^\perp}) dy'_i
\]
to the space of distributions and the latter\footnote{We actually have to restrict the domain of $D^{\circ}_i$ so that the functions and their image are both compactly supported.} corresponds to the division by the function $\xi_i$ (the dual variable of $y_i$) on the Fourier space. 
\end{proof}


\bibliography{mybib}
\bibliographystyle{amsplain} 

\end{document}